\documentclass[reqno]{amsart}
\usepackage{amsfonts}

\newtheorem{theorem}{Theorem}
\theoremstyle{plain}

\newtheorem{corollary}{Corollary}

\newtheorem{definition}{Definition}

\newtheorem{lemma}{Lemma}

\newtheorem{proposition}{Proposition}
\newtheorem{remark}{Remark}

\numberwithin{equation}{section}
 \numberwithin{theorem}{section}
 \numberwithin{proposition}{section}
 \numberwithin{remark}{section}
 \numberwithin{definition}{section}
 \numberwithin{lemma}{section}
 \numberwithin{corollary}{section}
 \numberwithin{example}{section}
 \numberwithin{claim}{section}
\begin{document}
\title[Reaction-diffusion equations in fractured porous media]{%
Homogenization of reaction-diffusion equations in fractured porous media}
\author{Hermann Douanla}
\address{Hermann Douanla, Department of Mathematics, University of Yaounde
1, P.O. Box 812, Yaounde, Cameroon}
\email{hdouanla@gmail.com}
\author{Jean Louis Woukeng}
\address{Jean Louis Woukeng, Department of Mathematics and Computer Science,
University of Dschang, P.O. Box 67, Dschang, Cameroon}
\email{jwoukeng@yahoo.fr}
\date{May, 2015}

\begin{abstract}
The paper deals with the homogenization of reaction-diffusion equations with
large reaction terms in a multi-scale porous medium. We assume that the
fractures and pores are equidistributed and that the coefficients of the
equations are periodic. Using the multi-scale convergence method, we derive a homogenization result whose limit problem is defined on a fixed domain and is of convection-diffusion-reaction type.
\end{abstract}

\maketitle

\section{Introduction\label{S1}}

Our aim is to investigate, by mean of mathematical
homogenization techniques, the diffusion phenomenon in a multi-scale porous
medium. The medium consists of a connected network made of pores and
fractures which are equidistributed, and the diffusion process is modelled
by a semilinear reaction-diffusion equation with a large reaction term.

To be more precise, we consider a diffusion process modelled by the
following boundary value problem: 
\begin{equation}
\left\{ \begin{aligned} \rho \left( \frac{x}{\varepsilon }\right)
\frac{\partial u_{\varepsilon }}{\partial t}= \text{div}\left( A\left(
\frac{x}{\varepsilon },\frac{t}{\varepsilon ^{2}}\right) \nabla
u_{\varepsilon }\right) +\frac{1}{\varepsilon }g\left( \frac{x}{\varepsilon
},\frac{t}{\varepsilon ^{2}},u_{\varepsilon }\right) &\quad \text{ in }\ \
Q_{T}^{\varepsilon }=\Omega ^{\varepsilon }\times (0,T), \\ A\left(
\frac{x}{\varepsilon },\frac{t}{\varepsilon ^{2}}\right) \nabla
u_{\varepsilon }\cdot \nu =0 &\quad \text{ on }\ \ (\partial
\Omega^{\varepsilon}\setminus \partial \Omega)\times \left( 0,T\right), \\
u_{\varepsilon }=0 & \quad\text{ on }\ \ (\partial \Omega^{\varepsilon}\cap
\partial \Omega) \times \left( 0,T\right), \\ u_{\varepsilon }(x,0)=u^{0}(x)
& \quad \text{ in }\ \ \Omega ^{\varepsilon }, \end{aligned}\right.
\label{eq1}
\end{equation}%
where $T>0$ is a fixed real number representing the final time of the
process, $\Omega ^{\varepsilon }$ is a fractured porous domain in which the
process occurred and whose structure follows in the lines below (see \cite%
{Wou15}).

Let $\Omega $ be a bounded domain in $\mathbb{R}^{N}$ ($N\geq 3$) locally
located on one side of its Lipschitz continuous boundary $\partial \Omega $.
Let $Y=(0,1)^{N}$ be the unit cell in $\mathbb{R}^{N}$ and put $Y=\overline{Y%
}_{c}\cup Y_{m}$ where $Y_{m}$ and $Y_{c}$ are two disjoint open connected
sets representing the local structure of the porous matrix and the cracks
(fissures), respectively. We assume that a periodic repetition of $Y_{m}$ in 
$\mathbb{R}^{N}$ is connected and has a Lipschitz continuous boundary. Next,
we set $Y_{m}=Z_{s}\cup \overline{Z}_{p}$ where $Z_{p}$ and $Z_{s}$ are two
disjoint open connected sets representing the local structure of the solid
part of the porous matrix and the pores, respectively. We assume that $Z_{p}$
and $Z_{s}$ have strictly positive Lebesgue measures and that $Z_{s}$ has a
Lipschitz continuous boundary. The fractured porous medium $\Omega
^{\varepsilon }$ is defined as follows. For $\varepsilon >0$, we set 
\begin{equation*}
G_{m}=\cup _{k\in \mathbb{Z}^{N}}(k+Y_{m})\ \ \text{ and }\ \ G_{c}=\mathbb{R%
}^{N}\setminus \overline{G}_{m},
\end{equation*}%
and 
\begin{equation*}
G_{s}=\cup _{k\in \mathbb{Z}^{N}}(k+Z_{s})\ \ \text{and}\ \
G_{p}=G_{m}\setminus \overline{G}_{s},
\end{equation*}%
and we define the pores space $\Omega _{p}^{\varepsilon }=\Omega \cap
\varepsilon ^{2}G_{p}$ (this include the pores crossing $\partial \Omega $
), the cracks space $\Omega _{c}^{\varepsilon }=\Omega \cap \varepsilon
G_{c} $ (this includes the cracks crossing $\partial \Omega $ ) and the
fractured porous medium as: 
\begin{equation*}
\Omega ^{\varepsilon }=\Omega \setminus (\Omega _{p}^{\varepsilon }\cup
\Omega _{c}^{\varepsilon }).
\end{equation*}%
We assume that both $\Omega ^{\varepsilon }$ and $\Omega _{p}^{\varepsilon
}\cup \Omega _{c}^{\varepsilon }$ are connected.

This being so, the $\varepsilon $-problem (\ref{eq1}) is constrained as
follows:

\begin{itemize}
\item[\textbf{A1}] \textbf{Uniform ellipticity}. The matrix $A(y,\tau
)=(a_{ij}(y,\tau ))_{1\leq i,j\leq N}\in (L^{\infty }(\mathbb{R}%
^{N+1}))^{N\times N}$ is real, symmetric, positive definite, i.e, there
exists $\Lambda >0$ such that 
\begin{equation*}
\begin{array}{l}
\left\Vert a_{ij}\right\Vert _{L^{\infty }(\mathbb{R}^{N+1})}\leq \Lambda
,\,\,1\leq i,j\leq N, \\ 
\sum_{ij=1}^{N}a_{ij}(y,\tau )\zeta _{i}\zeta _{j}\geq \Lambda
^{-1}\left\vert \zeta \right\vert ^{2}\text{ for all }(y,\tau )\in \mathbb{R}%
^{N+1},\zeta \in \mathbb{R}^{N}.%
\end{array}%
\end{equation*}

\item[\textbf{A2}] \textbf{Lipschitz continuity}. The function $g:\mathbb{R}%
^{N}\times \mathbb{R}\times \mathbb{R}\rightarrow \mathbb{R}$ satisfies the
following hypotheses. There exists $C>0$ such that for any $(y,\tau )\in 
\mathbb{R}^{N+1}$ and $u\in \mathbb{R}$, 
\begin{equation*}
\begin{array}{l}
\left\vert \partial _{u}g(y,\tau ,u)\right\vert \leq C \\ 
\left\vert \partial _{u}g(y,\tau ,u_{1})-\partial _{u}g(y,\tau
,u_{2})\right\vert \leq C\left\vert u_{1}-u_{2}\right\vert (1+\left\vert
u_{1}\right\vert +\left\vert u_{2}\right\vert )^{-1}.%
\end{array}%
\end{equation*}

\item[\textbf{A3}] $g(y,\tau ,0)=0$ for any $(y,\tau )\in \mathbb{R}^{N+1}$.

\item[\textbf{A4}] \textbf{Periodicity}. We assume that:

\begin{itemize}
\item[(i)] $g(\cdot ,\cdot ,u)\in \mathcal{C}_{\text{per}}(Y\times \mathcal{T%
})$ ($\mathcal{T}=(0,1)$) for any $u\in \mathbb{R}$ with $\int_{Y}g(y,\tau
,u)\,dy=0$ for all $(\tau ,u)\in \mathcal{T}\times \mathbb{R}$;

\item[(ii)] the functions $a_{ij}$ lie in $L_{\text{per}}^{2}(Y\times 
\mathcal{T})$ for all $1\leq i,j\leq N$;

\item[(iii)] the density function $\rho $ belongs to $C_{per}(Y)$ and
satisfies $\Lambda ^{-1}\leq \rho (y)\leq \Lambda $ for almost all $y\in 
\mathbb{R}^{N}$.
\end{itemize}
\end{itemize}

\begin{remark}
\label{r1}\emph{As a direct consequence of the periodicity and the zero mean
value hypothesis for the function }$g$\emph{\ (see precisely the first item
of the hypothesis A4 above), there exists a unique }$R(\cdot ,\cdot ,u)\in 
\mathcal{C}_{\text{per}}(Y\times T)$\emph{\ such that }$\Delta _{y}R(\cdot
,\cdot ,u)=g(\cdot ,\cdot ,u)$\emph{\ and }$\int_{Y}R(y,\tau ,u)\,dy=0$\emph{%
\ for all }$\tau $\emph{, }$u\in \mathbb{R}$\emph{. Moreover }$R(\cdot
,\cdot ,u)$\emph{\ is at least twice differentiable with respect to }$y$%
\emph{. Furthermore, on letting }$G=\nabla _{y}R$\emph{\ it follows from A2
and A3 that }%
\begin{equation}
\left\vert G(y,\tau ,u)\right\vert \leq C\left\vert u\right\vert \text{, }%
\left\vert \partial _{u}G(y,\tau ,u)\right\vert \leq C\text{,}  \label{eq8}
\end{equation}%
\emph{and} 
\begin{equation}
\left\vert \partial _{u}G(y,\tau ,u_{1})-\partial _{u}G(y,\tau
,u_{2})\right\vert \leq C\left\vert u_{1}-u_{2}\right\vert (1+\left\vert
u_{1}\right\vert +\left\vert u_{2}\right\vert )^{-1}.  \label{eq9}
\end{equation}
\end{remark}

\bigskip

The motivation for the problem (\ref{eq1}) arises from its applicability in
the area of modeling of flow and transport in fractured porous media related
to environmental and energy problems. In order to overcome difficulties
encountered in numerical simulations in multi-scale porous media, we need to
upscale such models, that is to find equivalent models by letting $%
\varepsilon \rightarrow 0$. This leads to model problems posed on a fixed
domain $Q=\Omega \times (0,T)$ with suitable boundary conditions, hence relatively easy to handle numerically.

Problem (\ref{eq1}) can also be viewed as modeling the flow of a single
phase compressible fluid in a fractured porous medium that obey nonlinear
Darcy law. In that case, $u_{\varepsilon }$ is the density of the fluid, $%
\rho (y)$ is the porosity of the medium while $A(y,\tau )$ is the
permeability of the medium. As the scale (size) of the fractures and that of the pores are separated (ratio of order $\varepsilon) 
$, we shall apply the multi-scale (or reiterated two-scale) convergence
techniques and the framework introduced in \cite{woukengaa} for the
upscaling of the same problem in a fixed domain.

The homogenization of parabolic equations has been widely investigated in the
literature. We quote some works similar to ours. In \cite{PR11} the
homogenization of parabolic monotone operator in periodically perforated
domain is considered. The problem they consider is degenerate and they use
the two-scale convergence method to achieve their goal. In \cite{NR02}, the
authors study the homogenization of a family of parabolic equations 
\begin{equation*}
\frac{\partial }{\partial t}b\left( \frac{x}{\varepsilon },u_{\varepsilon
}\right) -\text{div~}a(u_{\varepsilon },\nabla u_{\varepsilon })=f
\end{equation*}%
posed on a single periodically perforated domain $\Omega _{\varepsilon }$
with Dirichlet boundary conditions. In \cite{AGP07} is considered the
upscaling of a convection-diffusion equation in a perforated domain made of
holes periodically distributed. The homogenization limit for the diffusion
equation with nonlinear flux condition on the boundary of a periodically
perforated domain is studied in \cite{JNS10}. In \cite{DN04} the
homogenization of a semilinear parabolic equation in a periodically
perforated domain is considered. In \cite{showalter} the authors describe some diffusion models for fractured media. We also mention \cite{jaroudi} where the author used the $\Gamma$-convergence method associated with multi-scale convergence notions to get a limit law of an incompressible viscous flow in a porous medium with double porosity.

Taking into account the preceding review which is very far from being
exhaustive, we observe that the study of a problem like (\ref{eq1}) is relevant due to the geometry of the domain and we make use of the multi-scale convergence concept as the
homogenization method. This leads to the following main result of the paper, where in the passage to the limit (as $\varepsilon \to 0$), the large reaction term in the $\varepsilon$-problem generates a convection term and we get a limit problem of convection-diffusion-reaction type.

\begin{theorem}
Assume that the hypotheses \textbf{A1-A4} are in place and let $%
u_{\varepsilon }$ ($\varepsilon >0$) be the unique solution to \emph{(\ref%
{eq1})}. Then as $\varepsilon \rightarrow 0$ we have 
\begin{equation*}
u_{\varepsilon }\rightarrow u_{0}\quad \text{in }\ \ L^{2}(\Omega _{T}),
\end{equation*}%
where $u_{0}\in L^{2}(0,T;H_{0}^{1}(\Omega ))$ is the unique solution to 
\begin{equation*}
\left\{ \begin{aligned}
&|Z_s|\left(\int_{Y_m}\rho(y)\,dy\right)\frac{\partial u_0}{\partial t} =
\text{div}\, \left(\hat{A}(x,t)\nabla u_0\right) + \text{div}\, L_1(x,t,u_0)
- L_2(x,t,u_0)\cdot \nabla u_0 - L_3(x,t,u_0) \text{ in } \Omega_T\\ & u_0
=0 \qquad \text{on }\ \ \partial\Omega\times (0,T)\\ &u_0(x,0)= u^0(x)
\qquad \text{in }\ \ \Omega. \end{aligned}\right.
\end{equation*}
\end{theorem}
The coefficients  and operators in the theorem above are defined in Section~\ref{S4}.

The paper is organized as follows. The a priori estimates and compactness
results are formulated and proved in Section~\ref{S2}. In Section~\ref{S3},
we recall the concept of multi-scale convergence and prove some preliminary
results. Finally, Section~\ref{S4} deals with the passage to the limit and
the derivation of the macroscopic model for problem~(\ref{eq1}).

\section{A priori estimates and compactness result\label{S2}}

Throughout, $C$ denotes a generic constant independent of $\varepsilon $
that can change from one line to the next, the centered dot stands for
the Euclidean scalar product in $\mathbb{R}^{N}$ while the absolute value or
modulus is denoted by $|\cdot |$.

With the connectedness of $\Omega ^{\varepsilon }$ in mind, the space 
\begin{equation}
V_{\varepsilon }=\{u\in H^{1}(\Omega ^{\varepsilon }):u=0\text{ on }\partial
\Omega \cap \partial \Omega ^{\varepsilon }\}
\end{equation}%
is Hilbertian when endowed with the gradient norm, 
\begin{equation}
\Vert u\Vert _{V_{\varepsilon }}=\Vert \nabla u\Vert _{L^{2}(\Omega
^{\varepsilon })}\qquad (u\in V_{\varepsilon }).
\end{equation}%
Therefore, the Lipschitzity of the function $g(y,\tau ,\cdot )$ and the
positivity assumption on the density function $\rho $ readily imply (see
e.g., \cite{AL83, paronetto04}) the existence of a unique solution $%
u_{\varepsilon }\in L^{2}(0,T;V_{\varepsilon })\cap \mathcal{C}%
(0,T;L^{2}(\Omega ^{\varepsilon }))$ to the problem (\ref{eq1}). Moreover
the following uniform estimates hold.

\begin{lemma}
\label{l1} Assume that the hypotheses \textbf{A1}-\textbf{A4} are satisfied.
Then the following estimates hold true: 
\begin{equation}  \label{eq11}
\sup_{0\leq t\leq T}\|u_\varepsilon(t)\|^2_{L^2(\Omega^\varepsilon)}\leq C,
\end{equation}
\begin{equation}  \label{eq12}
\int_0^T\|\nabla u_\varepsilon(t)\|^2_{L^2(\Omega^\varepsilon)}\,dt \leq C,
\end{equation}
\begin{equation}  \label{eq121}
\left\|\rho^\varepsilon\frac{\partial u_\varepsilon}{\partial t}%
\right\|_{L^2(0,T;V^{\prime }_\varepsilon)}\leq C,
\end{equation}
where $C$ is a positive constant which does not depend on $\varepsilon$.
\end{lemma}

\begin{proof}
Let $t\in (0,T]$. Multiplying the first equation in (\ref{eq1}) by $%
u_\varepsilon$ and integrating over $\Omega^\varepsilon\times (0,t)$ yields: 
\begin{equation}  \label{eq13}
\left\|(\rho^\varepsilon)^{\frac{1}{2}}u_\varepsilon(t)\right\|^2_{L^2(%
\Omega^\varepsilon)}-\left\|(\rho^\varepsilon)^{\frac{1}{2}%
}u^0\right\|^2_{L^2(\Omega^\varepsilon)}+2\int_0^t\!\!\!\int_{\Omega^%
\varepsilon}A^\varepsilon|\nabla u_\varepsilon(s)|^2
dxds=2\int_0^t\!\!\!\int_{\Omega^\varepsilon}\frac{1}{\varepsilon}%
g^\varepsilon(u_\varepsilon(s))u_\varepsilon(s)dxds.
\end{equation}
But Remark~\ref{r1} readily implies 
\begin{equation*}
\frac{1}{\varepsilon}g\left(\frac{x}{\varepsilon},\frac{t}{\varepsilon^2}%
,u_\varepsilon\right)=\text{div}\, G\left(\frac{x}{\varepsilon},\frac{t}{%
\varepsilon^2},u_\varepsilon\right)-\partial_r G\left(\frac{x}{\varepsilon},%
\frac{t}{\varepsilon^2},u_\varepsilon\right)\cdot \nabla u_\varepsilon,
\end{equation*}
which combined with (\ref{eq13}) leads to 
\begin{eqnarray*}
\left\|(\rho^\varepsilon)^{\frac{1}{2}}u_\varepsilon(t)\right\|^2_{L^2(%
\Omega^\varepsilon)}+2\int_0^t\!\!\!\int_{\Omega^\varepsilon}A^\varepsilon|%
\nabla u_\varepsilon(s)|^2dxds &\leq &\left\|(\rho^\varepsilon)^{\frac{1}{2}%
}u^0\right\|^2_{L^2(\Omega^\varepsilon)}-2\int_0^t\!\!\!\int_{\Omega^%
\varepsilon}G^\varepsilon(u_\varepsilon)\cdot\nabla u_\varepsilon dxds \\
& & -2\int_0^t\!\!\!\int_{\Omega^\varepsilon}(\partial_r
G^\varepsilon(u_\varepsilon)\cdot \nabla u_\varepsilon)u_\varepsilon dxds,
\end{eqnarray*}
where $G^\varepsilon(u_\varepsilon)=G\left(\frac{x}{\varepsilon},\frac{t}{%
\varepsilon^2},u_\varepsilon\right)$ and $\partial_r
G^\varepsilon(u_\varepsilon)=\frac{\partial}{\partial_r} G\left(\frac{x}{%
\varepsilon},\frac{t}{\varepsilon^2},u_\varepsilon\right)$. Making use of (%
\ref{eq8}), the ellipticity of the matrix $A$ and the boundedness of the
function $\rho$, we have 
\begin{equation}  \label{eq14}
\Lambda^{-1}\left\|u_\varepsilon(t)\right\|^2_{L^2(\Omega^\varepsilon)}+2%
\Lambda^{-1}\int_0^t\!\!\!\int_{\Omega^\varepsilon}|\nabla
u_\varepsilon(s)|^2dxds \leq
\Lambda\left\|u^0\right\|^2_{L^2(\Omega)}+4C\int_0^t\!\!\!\int_{\Omega^%
\varepsilon}|u_\varepsilon||\nabla u_\varepsilon|\,dxds.
\end{equation}
For any real number $\delta>0$, we have by Young's inequality, 
\begin{equation*}
4C\int_0^t\!\!\!\int_{\Omega^\varepsilon}|u_\varepsilon||\nabla
u_\varepsilon|\,dxds \leq
4C\delta\int_0^t\!\!\!\int_{\Omega^\varepsilon}|u_\varepsilon|^2\,dxds + 
\frac{C}{\delta}\int_0^t\!\!\!\int_{\Omega^\varepsilon}|\nabla
u_\varepsilon|^2\,dxds.
\end{equation*}
Choosing $\delta>0$ such that $\frac{1}{\Lambda}=\frac{C}{\delta}$, the
inequality (\ref{eq14}) yields: 
\begin{equation*}
\Lambda^{-1}\left\|u_\varepsilon(t)\right\|^2_{L^2(\Omega^\varepsilon)}+%
\Lambda^{-1}\int_0^t\|\nabla
u_\varepsilon(s)\|^2_{L^2(\Omega^\varepsilon)}\,dxds\leq
\Lambda\left\|u^0\right\|^2_{L^2(\Omega)}+4C\delta\int_0^t\left\|u_%
\varepsilon(s)\right\|^2_{L^2(\Omega^\varepsilon)}\,dxds,
\end{equation*}
which by means of the Gronwall's inequality first leads to (\ref{eq11}),
then to (\ref{eq12}).

As for (\ref{eq121}), it follows from (\ref{eq1}) that 
\begin{equation}
\left\Vert \rho ^{\varepsilon }\frac{\partial u_{\varepsilon }}{\partial t}%
\right\Vert _{L^{2}(0,T;V_{\varepsilon }^{\prime })}^{2}\leq
C\int_{0}^{T}\Vert \text{div}\,A^{\varepsilon }\nabla u_{\varepsilon }\Vert
_{V_{\varepsilon }^{\prime }}^{2}\,dt+C\int_{0}^{T}\left\Vert \frac{1}{%
\varepsilon }g^{\varepsilon }(u_{\varepsilon })\right\Vert _{V_{\varepsilon
}^{\prime }}^{2}\,dt.  \label{eq141}
\end{equation}%
On the one hand, (\ref{eq12}) and the boundedness of the matrix $A$ imply 
\begin{equation}
\int_{0}^{T}\Vert \text{div}\,A^{\varepsilon }\nabla u_{\varepsilon }\Vert
_{V_{\varepsilon }^{\prime }}\,dt\leq C.  \label{eq15}
\end{equation}%
On the other hand, we have 
\begin{equation*}
\left\Vert \frac{1}{\varepsilon }g^{\varepsilon }(u_{\varepsilon
})\right\Vert _{V_{\varepsilon }^{\prime }}=\sup_{\varphi \in V_{\varepsilon
},\Vert \varphi \Vert _{V_{\varepsilon }}=1}\left\vert \int_{\Omega
^{\varepsilon }}G^{\varepsilon }(u_{\varepsilon })\cdot \nabla \varphi
\,dx+\int_{\Omega ^{\varepsilon }}(\partial _{r}G^{\varepsilon
}(u_{\varepsilon })\cdot \nabla u_{\varepsilon })\varphi \,dx\right\vert ,
\end{equation*}%
which by means of the Poincar\'{e}'s inequality and (\ref{eq8}) yields 
\begin{eqnarray*}
\left\Vert \frac{1}{\varepsilon }g^{\varepsilon }(u_{\varepsilon
})\right\Vert _{V_{\varepsilon }^{\prime }} &\leq &C\sup_{\varphi \in
V_{\varepsilon },\Vert \varphi \Vert _{V_{\varepsilon }}=1}\left( \Vert
u_{\varepsilon }(t)\Vert _{L^{2}(\Omega ^{\varepsilon })}+\Vert \nabla
u_{\varepsilon }(t)\Vert _{L^{2}(\Omega ^{\varepsilon })}\Vert \nabla
\varphi \Vert _{L^{2}(\Omega ^{\varepsilon })}\right) \\
&\leq &C\left( \Vert u_{\varepsilon }(t)\Vert _{L^{2}(\Omega ^{\varepsilon
})}+\Vert \nabla u_{\varepsilon }(t)\Vert _{L^{2}(\Omega ^{\varepsilon
})}\right) \qquad \qquad t\in (0,T).
\end{eqnarray*}%
Therefore, making use of (\ref{eq11})-(\ref{eq12}) and the H\"{o}lder's
inequality we get 
\begin{equation}
\int_{0}^{T}\left\Vert \frac{1}{\varepsilon }g^{\varepsilon }(u_{\varepsilon
})\right\Vert _{V_{\varepsilon }^{\prime }}^{2}dt\leq C.  \label{eq151}
\end{equation}%
We combine (\ref{eq141}), (\ref{eq15}) and (\ref{eq151}) to get (\ref{eq121}%
).
\end{proof}

The next result relies on the following classical extension property (see
e.g., \cite{Acerbi1992}).

\begin{proposition}
\label{p1}For any $\varepsilon>0$, there exists a bounded linear operator $%
P_\varepsilon$ from $V_\varepsilon$ into $H^1_0(\Omega)$ such that for any $%
u\in V_\varepsilon$ we have: 
\begin{eqnarray}
&P_\varepsilon u = u\ \ \text{ in }\ \ \Omega^\varepsilon,  \label{eq6} \\
&\|P_\varepsilon u\|_{H^1_0(\Omega)}\leq C \|u\|_{V_\varepsilon},
\label{eq7}
\end{eqnarray}
where $C$ is a positive constant independent of $\varepsilon$.
\end{proposition}

For a function $u\in L^{2}(0,T;V_{\varepsilon })$ we define its extension $%
P_{\varepsilon }u$ as follows 
\begin{equation}
(P_{\varepsilon }u)(t)=P_{\varepsilon }(u(t))\quad \text{ a.e. }t\in (0,T),
\end{equation}%
and $P_{\varepsilon }u\in L^{2}(0,T;H_{0}^{1}(\Omega ))$.

Bearing this in mind and owing to Proposition \ref{p1} (see precisely (\ref%
{eq7})), we have the following corollary.

\begin{corollary}
\label{c1} Under the hypotheses of Lemma \emph{\ref{l1}}, we have the
following uniform estimate 
\begin{equation}
\Vert P_{\varepsilon }u_{\varepsilon }\Vert _{L^{2}(0,T;H_{0}^{1}(\Omega
))}\leq C  \label{eq16}
\end{equation}%
where $C>0$ is a positive constant independent of $\varepsilon $ and where $%
P_{\varepsilon }$ is the extension operator defined in Proposition \emph{\ref%
{l1}}.
\end{corollary}

The next estimate requires some preliminaries. We define $R_\varepsilon:
H^1_0(\Omega)\to V_\varepsilon$ by $R_\varepsilon u =
u|_{\Omega_\varepsilon} $ for $u\in H^1_0(\Omega)$ (where $%
u|_{\Omega_\varepsilon}$ denotes the restriction of $u$ to $%
\Omega_\varepsilon$). Then, $R_\varepsilon$ is continuous since 
\begin{equation*}
\|R_\varepsilon u\|_{V_\varepsilon}\leq \|u\|_{H^1_0(\Omega)} \quad \text{
for }\ \ u\in H^1_0(\Omega).
\end{equation*}
We recall that the adjoint $R^*_\varepsilon: V^{\prime }_\varepsilon\to
H^{-1}(\Omega)$ of $R_\varepsilon$ satisfies, for all $v\in V^{\prime
}_\varepsilon$ and $\psi \in H^1_0(\Omega)$, 
\begin{equation*}
\langle R^*_\varepsilon v, \varphi \rangle=\langle v, R_\varepsilon \varphi
\rangle,
\end{equation*}
where the brackets on the left hand side denote the duality pairing between
the spaces $H^{-1}(\Omega)$ and $H^1_0(\Omega)$ while those on the right
hand side denote the duality pairing between $V^{\prime }_\varepsilon$ and $%
V_\varepsilon$. It is straightforward that 
\begin{equation}  \label{eq161}
R^*_\varepsilon u = \chi_{\Omega^\varepsilon}\, u \quad \text{ for }\ \ u\in
L^2(\Omega^\varepsilon\times (0,T)).
\end{equation}
Indeed, for any $\varphi\in L^2(0,T; H^1_0(\Omega))$, we have 
\begin{eqnarray*}
\langle R^*_\varepsilon u, \varphi\rangle &=& \langle u, R_\varepsilon
\varphi\rangle =
\int_0^T\!\!\int_{\Omega^\varepsilon}u\,(\varphi|_{\Omega^\varepsilon})dx\,dt
\\
&=&
\int_0^T\!\!\int_{\Omega}\chi_{\Omega^\varepsilon}(u\,\varphi)dx\,dt=%
\int_0^T\!\!\int_{\Omega}(\chi_{\Omega^\varepsilon}\,u)\,\varphi\,dx\,dt.
\end{eqnarray*}
By the way, it is worth noticing that combining (\ref{eq161}) and
Proposition \ref{p1} (see precisely (\ref{eq6}) therein), we have 
\begin{equation}  \label{eq17}
R^*_\varepsilon u = \chi_{\Omega^\varepsilon}( P_\varepsilon u) \quad \text{
for } \ \ u\in L^2(0,T;V_\varepsilon).
\end{equation}
Likewise, one can easily check that, for any $u\in L^2(0,T; V_\varepsilon)$
with $\frac{\partial u}{\partial t}\in L^2(0,T;V^{\prime }_\varepsilon)$, we
have

\begin{equation}  \label{eq171}
R^*_\varepsilon\left(\frac{\partial u_\varepsilon}{\partial t}\right) =\frac{%
\partial ( R^*_\varepsilon u_\varepsilon)}{\partial t}.
\end{equation}

We are now in a position to formulate another estimate.

\begin{lemma}
\label{l2} There exists a constant $C$ independent of $\varepsilon$ such
that 
\begin{equation}  \label{eq18}
\left\|(\rho^\varepsilon\chi_{\Omega^\varepsilon})\frac{\partial
(P_\varepsilon u_\varepsilon)}{\partial t}\right\|_{L^2(0,T;H^{-1}(\Omega))}%
\leq C.
\end{equation}
\end{lemma}

\begin{proof}
We first prove that there exists a constant $C$ independent of $\varepsilon$
such that 
\begin{equation}  \label{eq19}
\left\|R^*_\varepsilon\left(\rho^\varepsilon\frac{\partial u_\varepsilon}{%
\partial t}\right)\right\|_{L^2(0,T;H^{-1}(\Omega))}\leq C.
\end{equation}
To do this, let $\varphi$ be arbitrarily fixed in $L^2(0,T;H^1_0(\Omega))$.
We have 
\begin{eqnarray*}
\left|\left\langle R^*_\varepsilon\left(\rho^\varepsilon\frac{\partial
u_\varepsilon}{\partial t}\right),
\varphi\right\rangle\right|&=&\left|\left\langle \rho^\varepsilon\frac{%
\partial u_\varepsilon}{\partial t},
R_\varepsilon\varphi\right\rangle\right|=\left|\int_0^T\left\langle
\rho^\varepsilon\frac{\partial u_\varepsilon}{\partial t},
R_\varepsilon\varphi\right\rangle_{V^{\prime }_\varepsilon,V_\varepsilon}
dt\right| \\
& \leq & \left\|\rho^\varepsilon\frac{\partial u_\varepsilon}{\partial t}%
\right\|_{L^2(0,T, V^{\prime
}_\varepsilon)}\|R_\varepsilon\varphi\|_{L^2(0,T;V_\varepsilon)} \\
&\leq & C\|R_\varepsilon\varphi\|_{L^2(0,T;V_\varepsilon)}\qquad\qquad
\qquad (\text{see (\ref{eq121})}) \\
&\leq & C\|\varphi\|_{L^2(0,T;H^1_0(\Omega))}.
\end{eqnarray*}
Having done this, it remains to prove that 
\begin{equation}  \label{eq20}
R^*_\varepsilon\left(\rho^\varepsilon\frac{\partial u_\varepsilon}{\partial t%
}\right)=\rho^\varepsilon\chi_{\Omega^\varepsilon}\frac{\partial
(P_\varepsilon u_\varepsilon)}{\partial t}.
\end{equation}
But with (\ref{eq17}) and (\ref{eq171}) in mind, it is easy to see that 
\begin{equation*}
R^*_\varepsilon\left(\rho^\varepsilon\frac{\partial u_\varepsilon}{\partial t%
}\right) =\rho^\varepsilon R^*_\varepsilon\left(\frac{\partial u_\varepsilon%
}{\partial t}\right) = \rho^\varepsilon\frac{\partial ( R^*_\varepsilon
u_\varepsilon)}{\partial t} = \rho^\varepsilon\frac{\partial (
\chi_{\Omega^\varepsilon} P_\varepsilon u_\varepsilon)}{\partial t} =
\rho^\varepsilon \chi_{\Omega^\varepsilon}\frac{\partial (P_\varepsilon
u_\varepsilon)}{\partial t},
\end{equation*}
and the proof is completed.
\end{proof}

The following compactness result will be the starting point of our
homogenization process.

\begin{theorem}
\label{t1} Assume that the sequence $(\rho^\varepsilon\chi_{\Omega^%
\varepsilon})_{\varepsilon>0}$ weakly $*$- converges in $L^{\infty}(\Omega)$%
, as $\varepsilon\to 0$, to some real function that is different from zero
almost everywhere in $\Omega$. Then the sequence $(P_\varepsilon
u_\varepsilon )_{\varepsilon>0}$ is relatively compact in $L^2
(0,T;L^2(\Omega))$.
\end{theorem}

\begin{proof}
This is a direct consequence of the convergence hypothesis on the sequence $%
(\rho^\varepsilon\chi_{\Omega^\varepsilon})_{\varepsilon>0}$, Corollary \ref%
{c1} and Lemma \ref{l2}, by using \cite[ Theorem 2.3 and Remark 2.5]{ADP05}.
\end{proof}

\section{Multi-scale convergence and preliminary convergence results\label%
{S3}}

We recall the definition and some compactness results of the multi-scale
convergence theory~\cite{AB96, CMA, CPAA}. We also introduce our functional
setting and adapt some results of the multi-scale convergence method to our
framework. We finally prove some preliminary convergence results needed in
the homogenization process of the problem under consideration. We introduce
the following notations: $\Omega _{T}=\Omega \times (0,T)$ and $\mathcal{T}%
=(0,1)$.

\subsection{Multi-scale convergence method}

\begin{definition}
\label{d1}

\begin{itemize}
\item[(i)] A sequence $(u_\varepsilon)_{\varepsilon>0}\subset L^2(\Omega_T)$
is said to weakly multi-scale converge towards $u_0\in L^2(\Omega_T\times
Y\times Z\times \mathcal{T})$, and denoted $u_\varepsilon\xrightarrow{w-ms}
u_0$ in $L^2(\Omega_T)$, if as $\varepsilon\to 0$, 
\begin{equation}  \label{eq21}
\int_{\Omega_T}u_\varepsilon(x,t)\varphi(x,t,\frac{x}{\varepsilon},\frac{x}{%
\varepsilon^2},\frac{t}{\varepsilon^2})\,dxdt\to\iiiint_{\Omega_T\times
Y\times Z\times \mathcal{T}}u_0(x,t,y,z,\tau)\varphi(x,t,y,z,\tau)%
\,dxdtdydzd\tau
\end{equation}
for all $\varphi\in L^2(\Omega_T;\mathcal{C}_{per}(Y\times Z\times \mathcal{T%
}))$. \vspace{.4cm}

\item[(ii)] A sequence $(u_\varepsilon)_{\varepsilon>0}\subset L^2(\Omega_T)$
is said to strongly multi-scale converge towards $u_0\in L^2(\Omega_T\times
Y\times Z\times \mathcal{T})$, and denoted $u_\varepsilon\xrightarrow{s-ms}
u_0$ in $L^2(\Omega_T)$, if it multi scale converges weakly to $u_0$ in $%
L^2(\Omega_T\times Y\times Z\times \mathcal{T})$ and further satisfies

\begin{equation}  \label{eq22}
\|u_{\varepsilon}\|_{L^2(\Omega_T)} \to \|u_0\|_{L^2(\Omega_T\times Y\times
Z\times\mathcal{T})} \quad \text{as} \quad \varepsilon \to 0.
\end{equation}
\end{itemize}
\end{definition}

\begin{remark}
\label{r2}

\begin{itemize}
\item[(i)] \emph{Let }$u\in L^{2}(\Omega _{T};\mathcal{C}_{per}(Y\times
Z\times \mathcal{T}))$\emph{\ and define }$u^{\varepsilon }:\Omega
_{T}\rightarrow \mathbb{R}$\emph{\ by }$u^{\varepsilon }(x,t)=u(x,t,\frac{x}{%
\varepsilon },\frac{x}{\varepsilon ^{2}},\frac{t}{\varepsilon ^{2}})$\emph{%
,\ \ for }$\varepsilon >0$\emph{\ and }$(x,t)\in \Omega _{T}$\emph{. Then} $%
u^{\varepsilon }\xrightarrow{w-ms}u$\emph{\ and }$u^{\varepsilon }%
\xrightarrow{s-ms}u$\emph{\ in }$L^{2}(\Omega _{T})$\emph{\ as }$\varepsilon
\rightarrow 0$\emph{. We also have} 
\begin{equation*}
u^{\varepsilon }\rightarrow \widetilde{u}\text{\emph{\ in }}L^{2}(\Omega
_{T})\text{\emph{-weak\ \ as }}\varepsilon \rightarrow 0\text{\emph{, with }}%
\widetilde{u}(x,t)=\iiint_{Y\times Z\times \mathcal{T}}u({\cdot ,\cdot
,y,z,\tau })\,dydzd\tau .
\end{equation*}

\item[(ii)] \emph{If a sequence }$(u_{\varepsilon })_{\varepsilon >0}\subset
L^{2}(\Omega _{T})$\emph{\ multi-scale converges weakly in }$L^{2}(\Omega
_{T})$\emph{\ to some }$u_{0}\in L^{2}(\Omega _{T}\times Y\times Z\times 
\mathcal{T})$\emph{, in the sense of Definition~\ref{d1}, then (\ref{eq21})
still holds for }$\varphi \in \mathcal{C}(\overline{\Omega }%
_{T};L_{per}^{\infty }(Y\times Z\times \mathcal{T}))$.

\item[(iii)] \emph{Let }$u\in \mathcal{C}(\overline{\Omega }%
_{T};L_{per}^{\infty }(Y\times Z\times \mathcal{T}))$\emph{\ and define }$%
u^{\varepsilon }$\emph{\ like in (i) above. Then }$u^{\varepsilon }%
\xrightarrow{w-ms}u$\emph{\ in }$L^{2}(\Omega _{T})$\emph{\ as }$\varepsilon
\rightarrow 0$\emph{.}
\end{itemize}
\end{remark}

The following two compactness results are the cornerstones of the
multi-scale convergence theory.

\begin{theorem}
\label{t2} Any bounded sequence in $L^2(\Omega_T)$ admits a weakly
multi-scale convergent subsequence.
\end{theorem}

Let $E$ be an ordinary sequence of real number converging to zero with $%
\varepsilon$.

\begin{theorem}
\label{t3} Let ($u_{\varepsilon})_{\varepsilon\in E}$ be a bounded sequence
in $L^2(0,T;H^1_0(\Omega))$. There exist a subsequence $E^{\prime }$ of $E$
and a triplet $(u_0,u_1,u_2)\in L^2(0,T;H^1_0(\Omega))\times
L^2(\Omega_T;L^2(\mathcal{T};H^1_{per}(Y)))\times L^2(\Omega_T;L^2(Y\times%
\mathcal{T};H^1_{per}(Z)))$ such that, as $E^{\prime }\ni\varepsilon\to 0$, 
\begin{eqnarray}
u_\varepsilon & \to & u_0 \qquad \text{ in }\quad L^2(0,T;H^1_0(\Omega))%
\text{-weak}  \label{eq23} \\
\notag \\
\frac{\partial u_\varepsilon}{\partial x_i} & \xrightarrow{w-ms} & \frac{%
\partial u_0}{ \partial x_i}+\frac{\partial u_1}{\partial y_i}+\frac{%
\partial u_2}{\partial z_i} \quad \text{ in }\quad L^2(\Omega_T)\quad (1\leq
j\leq N).  \label{eq24}
\end{eqnarray}
\end{theorem}

We need to tailor Theorem~\ref{t3} according to our needs. The functions $%
u_1 $ and $u_2$ in Theorem~\ref{t3} are unique up to additive function of
variables $x,t\tau$ and $x,t,y,\tau$, respectively. It is crucial to fix the
choice of $u_1$. We introduce the space 
\begin{equation*}
H^1_{\rho}(Y_m)=\{ u\in H^1_{per}(Y) : \int_{Y_m}\rho(y)u(y)\,dy=0\},
\end{equation*}
which is a closed subspace of $H^1_{per}(Y)$ since it is the kernel of the
bounded linear functional $u\mapsto \int_{Y_m}\rho(y)u(y)\,dy$ defined on $%
H^1_{per}(Y)$. The version of Theorem~\ref{t3} that will be used in the
sequel formulates as follows.

\begin{theorem}
\label{t4} Let ($u_{\varepsilon})_{\varepsilon\in E}$ be a bounded sequence
in $L^2(0,T;H^1_0(\Omega))$. There exist a subsequence $E^{\prime }$ of $E$
and a triplet $(u_0,u_1,u_2)\in L^2(0,T;H^1_0(\Omega))\times
L^2(\Omega_T;L^2(\mathcal{T};H^1_{\rho}(Y_m)))\times L^2(\Omega_T;L^2(Y\times%
\mathcal{T};H^1_{per}(Z)))$ such that, as $E^{\prime }\ni\varepsilon\to 0$, 
\begin{eqnarray}
u_\varepsilon & \to & u_0 \qquad \text{ in }\quad L^2(0,T;H^1_0(\Omega))%
\text{-weak}  \label{eq25} \\
\notag \\
\frac{\partial u_\varepsilon}{\partial x_i} & \xrightarrow{w-ms} & \frac{%
\partial u_0}{ \partial x_i}+\frac{\partial u_1}{\partial y_i}+\frac{%
\partial u_2}{\partial z_i} \quad \text{ in }\quad L^2(\Omega_T)\quad (1\leq
j\leq N).  \label{eq26}
\end{eqnarray}
\end{theorem}

\begin{proof}
The proof is similar to the proof of \cite[Theorem~2.5]{douanlaaa}.
\end{proof}

The following weak-strong convergence result (see \cite[Theorem 6]{SW11} for
its proof) and its corollary are worth recalling since they will be used in
the sequel.

\begin{theorem}
\label{t5} Let $(u_\varepsilon)_{\varepsilon\in E}\subset L^2(\Omega_T)$ and 
$(v_\varepsilon)_{\varepsilon\in E}\subset L^2(\Omega_T)$ be two sequences
such that $u_\varepsilon \xrightarrow{w-ms} u_0$ \ and \ $v_\varepsilon %
\xrightarrow{s-ms} v_0$ in $L^2(\Omega_T)$ with $u_0,v_0\in
L^2(\Omega_T\times Y\times Z\times \mathcal{T})$. Then $u_\varepsilon
v_\varepsilon \xrightarrow{w-ms} u_0 v_0$ in $L^1(\Omega_T)$.
\end{theorem}

\begin{corollary}
\label{c2} Let $(u_\varepsilon)_{\varepsilon\in E}\subset L^2(\Omega_T)$ and 
$(v_\varepsilon)_{\varepsilon\in E}\subset L^2(\Omega_T)\cap
L^\infty(\Omega_T)$ be two sequences such that $u_\varepsilon %
\xrightarrow{w-ms} u_0$ \ and \ $v_\varepsilon \xrightarrow{s-ms} v_0$ in $%
L^2(\Omega_T)$ with $u_0,v_0\in L^2(\Omega_T\times Y\times Z\times \mathcal{T%
})$. Assume further that $(v_\varepsilon)_{\varepsilon\in E}$ is bounded in $%
L^\infty(\Omega_T)$. Then $u_\varepsilon v_\varepsilon \xrightarrow{w-ms}
u_0 v_0$ in $L^2(\Omega_T)$.
\end{corollary}

\subsection{Preliminary convergence results}

We start this subsection by studying the limiting behavior of the sequence $%
(\chi _{\Omega ^{\varepsilon }})_{\varepsilon >0}$ as $\varepsilon
\rightarrow 0$. To do this, we first express the characteristic function of $%
\Omega ^{\varepsilon }$ in $\Omega $, in terms of those of $Y_{m}$ and $%
Z_{s} $. Denoting by $\chi _{c}^{\varepsilon }$ and $\chi _{p}^{\varepsilon
} $ the characteristic functions of $\Omega _{c}^{\varepsilon }$ and $\Omega
_{p}^{\varepsilon }$, respectively, it appears that 
\begin{eqnarray*}
\chi _{c}^{\varepsilon }(x) &=&\chi _{G_{c}}\left( \frac{x}{\varepsilon }%
\right) =\chi _{Y_{c}}\left( \frac{x}{\varepsilon }\right) \qquad \qquad (%
\text{ by }Y\text{-periodicity)} \\
\chi _{p}^{\varepsilon }(x) &=&\left( 1-\chi _{G_{c}}\left( \frac{x}{%
\varepsilon }\right) \right) \chi _{G_{p}}\left( \frac{x}{\varepsilon ^{2}}%
\right) \\
&=&\left( 1-\chi _{Y_{c}}\left( \frac{x}{\varepsilon }\right) \right) \chi
_{Z_{p}}\left( \frac{x}{\varepsilon ^{2}}\right) \qquad \qquad \text{ (by }Y%
\text{ and Z-periodicity)}.
\end{eqnarray*}%
Hence 
\begin{eqnarray*}
\chi _{\Omega ^{\varepsilon }}(x) &=&1-(\chi _{c}^{\varepsilon }(x)+\chi
_{p}^{\varepsilon }(x))\qquad \qquad \qquad \qquad \qquad (x\in \Omega ) \\
&=&1-\left[ \chi _{Y_{c}}\left( \frac{x}{\varepsilon }\right) +\left( 1-\chi
_{Y_{c}}\left( \frac{x}{\varepsilon }\right) \right) \chi _{Z_{p}}\left( 
\frac{x}{\varepsilon ^{2}}\right) \right] \\
&=&\left( 1-\chi _{Y_{c}}\left( \frac{x}{\varepsilon }\right) \right) \left(
1-\chi _{Z_{p}}\left( \frac{x}{\varepsilon ^{2}}\right) \right) \\
&=&\chi _{Y_{m}}\left( \frac{x}{\varepsilon }\right) \chi _{Z_{s}}\left( 
\frac{x}{\varepsilon ^{2}}\right) .
\end{eqnarray*}%
But, $\chi _{Y_{m}}(\cdot )\otimes \chi _{Z_{s}}(\cdot )\in L_{per}^{\infty
}(Y\times Z)$ so that according to $(iii)$ of Remark~\ref{r2}, 
\begin{equation*}
\chi _{\Omega ^{\varepsilon }}\xrightarrow{w-ms}\chi _{Y_{m}}\otimes \chi
_{Z_{s}}\quad \text{ in }\quad L^{2}(\Omega _{T}),
\end{equation*}%
where the tensor product $\chi _{Y_{m}}\otimes \chi _{Z_{s}}$ is defined by $%
\left( \chi _{Y_{m}}\otimes \chi _{Z_{s}}\right) (y,z)=\chi _{Y_{m}}(y)\chi
_{Z_{s}}(z),\ \ \ ((y,z)\in \mathbb{R}^{N}\times \mathbb{R}^{N}$). We have
proved the following result.

\begin{proposition}
\label{p2} As $\varepsilon\to 0$, the characteristic function $%
\chi_{\Omega^\varepsilon}$ of $\Omega^\varepsilon$ multi-scale converges
weakly in $L^2(\Omega_T)$ to $\chi_{Y_{m}}\otimes \chi _{Z_{p}}$.
\end{proposition}

We now proceed and recall properties of some functional spaces we will use.
The topological dual of $H^1_\rho(Y_m)$ is denoted in the sequel by $%
(H^1_\rho(Y_m))^{\prime }$ while $L^2_\rho(Y_m)$ stands for the space of
functions $u\in L^2_{per}(Y)$ satisfying $\int_{Y_m}\rho(y)u(y)\,dy=0$. We
first recall that, since the space $H^1_\rho(Y_m)$ is densely embedded in $%
L^2_\rho(Y_m)$, the following continuous embeddings hold: 
\begin{equation*}
H^1_\rho(Y_m)\subset L^2_\rho(Y_m)\subset (H^1_\rho(Y_m))^{\prime }.
\end{equation*}
We also recall that the topological dual of $L^2(\mathcal{T};H^1_\rho(Y_m))$
is $L^2(\mathcal{T};(H^1_\rho(Y_m))^{\prime })$. This readily follows from
the reflexiveness of the space $H^1_\rho(Y_m)$. We denote the duality
pairing between $H^1_\rho(Y_m)$ and $(H^1_\rho(Y_m))^{\prime }$ by $%
(\cdot,\cdot)$, and that of $L^2(\mathcal{T};H^1_\rho(Y_m))$ and $L^2(%
\mathcal{T};(H^1_\rho(Y_m))^{\prime })$ by $[\cdot,\cdot]$. Thus, we have 
\begin{equation*}
(u,v)=\int_{Y}u(y)v(y)\,dy
\end{equation*}
for $u\in L^2_\rho(Y_m)$ and $v\in H^1_\rho(Y_m)$, and 
\begin{equation*}
[u,v]=\int_0^1(u(\tau),v(\tau))\,d\tau=\int_0^1\int_{Y}u(y,\tau)v(y,\tau)%
\,dy\,d\tau
\end{equation*}
for $u\in L^2(\mathcal{T};H^1_\rho(Y_m))$ and $v\in L^2(\mathcal{T}%
;(H^1_\rho(Y_m))^{\prime })$. Furthermore, let $\mathcal{D}_\rho(Y_m)$
stands for the space of functions $u\in \mathcal{D}_{per}(Y)$ with $%
\int_{Y_m}\rho(y)u(y)\,dy=0$. Owing to the fact that the space $\mathcal{D}%
_\rho(Y_m)$ is dense in $H^1_\rho(Y_m)$ the following result holds (see e.g. 
\cite[Lemma 2 and Lemma 3]{woukengaa}).

\begin{theorem}
Let $u\in \mathcal{D}_{per}^{\prime }(Y\times \mathcal{T})$ and assume that $%
u$ is continuous on $\mathcal{D}_{\rho }(Y_{m})\otimes \mathcal{D}_{per}(%
\mathcal{T})$ endowed with the $L_{per}^{2}(\mathcal{T};H_{\rho
}^{1}(Y_{m})) $-norm. Then $u\in L_{per}^{2}(\mathcal{T};(H_{\rho
}^{1}(Y_{m}))^{\prime })$, and further 
\begin{equation*}
\langle u,\varphi \rangle =\int_{0}^{1}(u(\tau ),\varphi (\cdot ,\tau
))\,d\tau
\end{equation*}%
for all $\varphi \in \mathcal{D}_{\rho }(Y_{m})\otimes \mathcal{D}_{per}(%
\mathcal{T})$, where $\langle \cdot ,\cdot \rangle $ denotes the duality
pairing between $\mathcal{D}_{per}^{\prime }(Y\times \mathcal{T})$ and $%
\mathcal{D}_{per}(Y\times \mathcal{T})$, whereas the right-hand side is the
product of $u$ and $\varphi $ in the duality between $L_{per}^{2}(\mathcal{T}%
;(H_{\rho }^{1}(Y_{m}))^{\prime })$ and $L_{per}^{2}(\mathcal{T};H_{\rho
}^{1}(Y_{m}))$.
\end{theorem}

We now define an operator 
\begin{eqnarray}  \label{eq261}
R:L^2_{per}(Y)&\to & L^2_{per}(Y)  \notag \\
u & \mapsto &\chi_{Y_m} \rho u.
\end{eqnarray}
It is clear that $R$ is a non-negative and linear bounded self-adjoint
operator. Using the positivity of the weight $\rho$ we prove that the kernel
of $R$ is defined by: 
\begin{equation*}
Ker(R)= \{u\in L^2_{per}(Y) : u=0 \ \ \text{a.e. in }\ Y_m\}.
\end{equation*}
We denote by $Ker(R)^{\bot}$ the orthogonal of the kernel of $R$ in $%
L^2_{per}(Y)$ while $L^2_0(Y_m)$ stands for the completion of $Ker(R)^\bot $
with respect to the norm 
\begin{equation*}
\|u\|_+ = \|\chi_{Y_m}\rho^\frac{1}{2}u\|_{L^2_{per}(Y)}\quad\qquad (u\in
Ker(R)^{\bot}).
\end{equation*}
We denote by $P$ the orthogonal projection from $L^2_{per}(Y)$ onto $%
L^2_0(Y_m)$. We recall that for $u\in L^2_{per}(\mathcal{T};L^2_{per}(Y))$
we define $Ru$ and $Pu$ by 
\begin{equation*}
(Ru)(\tau)=R(u(\tau))\ \ \text{and }\ (Pu)(\tau)=P(u(\tau)) \ \ \ \text{for
a.e. } \tau\in (0,1).
\end{equation*}
Considered as an unbounded operator on $L^2_{per}(\mathcal{T};H^1_\rho(Y_m))$%
, the domain of $R^{\prime }\equiv\chi_{Y_m}\rho\frac{\partial}{\partial \tau%
}$ is 
\begin{equation*}
\mathcal{W}=\left\{u\in L^2_{per}(\mathcal{T};H^1_\rho(Y_m)) : \chi_{Y_m}\rho%
\frac{\partial u}{\partial \tau}\in L^2_{per}(\mathcal{T};(H^1_\rho(Y_m))^{%
\prime }) \right\}.
\end{equation*}
We endow $\mathcal{W}$ with its natural norm 
\begin{equation*}
\|u\|_\mathcal{W}=\|u\|_{L^2_{per}(\mathcal{T};H^1_\rho(Y_m)} +
\left\|\chi_{Y_m}\rho\frac{\partial u}{\partial \tau} \right\|_{L^2_{per}(%
\mathcal{T};(H^1_\rho(Y_m))^{\prime })} \qquad\qquad (u\in \mathcal{W}),
\end{equation*}
and recall an important result (see e.g., \cite{pankovpankova, paronetto04}
) we will use in the sequel.

\begin{proposition}
\label{p3} The operator $P$ maps continuously $\mathcal{W}$ into $\mathcal{C}%
([0,1];L^2_0(Y_m))$, i.e., there exists a constant $c>0$ such that 
\begin{equation*}
\|Pu\|_{\mathcal{C}([0,1];L^2_0(Y_m))} = \sup_{0\leq \tau\leq
1}\left\|\chi_{Y_m} \rho^{\frac{1}{2}}u(\tau)\right\|_{L^2(Y)}\leq c \|u\|_%
\mathcal{W} \qquad \text{ for all } u\in\mathcal{W}.
\end{equation*}
Moreover, 
\begin{equation}  \label{eq262}
\left[\chi_{Y_m} \rho\frac{\partial u}{\partial \tau},v\right]=-\left[%
\chi_{Y_m} \rho\frac{\partial v}{\partial\tau},u\right]\qquad \text{ for all 
} u,v\in \mathcal{W}.
\end{equation}
\end{proposition}

\noindent We will also make use of the following convergence results in the
forthcoming homogenization process.

\begin{proposition}
\label{p4} Let $\varphi \in \mathcal{C}_{0}^{\infty }(Q_{T})\otimes \mathcal{%
D}_{per}(\mathcal{T})\otimes \mathcal{D}_{\rho }(Y_{m})$. Let $E^{\prime }$, 
$(u_{\varepsilon })_{\varepsilon \in E}$ and $(u_{0},u_{1})\in
L^{2}(0,T;H_{0}^{1}(\Omega ))\times L^{2}(\Omega _{T};L^{2}(\mathcal{T}%
;H_{\rho }^{1}(Y_{m})))$ be as in Theorem \ref{t4}. Then 
\begin{equation*}
\lim_{E\ni \varepsilon \rightarrow 0}\ \ \frac{1}{\varepsilon }%
\int_{Q_{T}^{\varepsilon }}u_{\varepsilon }(x,t)\rho (\frac{x}{\varepsilon }%
)\varphi (x,t,\frac{x}{\varepsilon },\frac{t}{\varepsilon ^{2}})dx\,dt\ \
=|Z_{s}|\int_{Q_{T}}[\chi _{Y_{m}}\,\rho \,u_{1}(x,t),\varphi (x,t)]dx\,dt,
\end{equation*}%
where $0<|Z_{s}|<1$ denotes the Lebesgue measure of the set $Z_{s}$.
\end{proposition}

\begin{proof}
We have 
\begin{eqnarray}
\frac{1}{\varepsilon}\int_{Q^\varepsilon_T}u_\varepsilon(x,t)\rho(\frac{x}{%
\varepsilon})\varphi(x,t,\frac{x}{\varepsilon},\frac{t}{\varepsilon^2})dx\,
dt &= &\frac{1}{\varepsilon}\int_{Q_T}u_\varepsilon(x,t)\chi_{\Omega^%
\varepsilon}(x)\rho(\frac{x}{\varepsilon})\varphi(x,t,\frac{x}{\varepsilon},%
\frac{t}{\varepsilon^2})dx\, dt  \notag \\
&=& \frac{1}{\varepsilon}\int_{Q_T}u_\varepsilon(x,t)\chi_{Y_m}(\frac{x}{%
\varepsilon})\chi_{Z_s}(\frac{x}{\varepsilon^2})\rho(\frac{x}{\varepsilon}%
)\varphi(x,t,\frac{x}{\varepsilon},\frac{t}{\varepsilon^2})dx\,dt \qquad 
\notag
\end{eqnarray}
Bearing in mind that 
\begin{equation*}
\int_{Y}\chi_{Y_m}(y)\rho(y)\varphi(x,t,y,\tau)dy = 0\ \ \text{ for all }\ \
(x,t,\tau)\in\mathbb{R}^N\times\mathbb{R}\times\mathbb{R}),
\end{equation*}
the same line of reasoning as in the proof of \cite[Theorem 2.3]{douanlaaa}
yields, as $E\ni\varepsilon\to 0$, 
\begin{eqnarray}
&&\frac{1}{\varepsilon}\int_{Q_T}u_\varepsilon(x,t)\chi_{Y_m}(\frac{x}{%
\varepsilon})\chi_{Z_s}(\frac{x}{\varepsilon^2})\rho(\frac{x}{\varepsilon}%
)\varphi(x,t,\frac{x}{\varepsilon},\frac{t}{\varepsilon^2})\,dxdt  \notag \\
& & \qquad \qquad\qquad\qquad\qquad\qquad\qquad\to \int_{Q_T\times Y\times
Z\times \mathcal{T}}\chi_{Y_m}\chi_{Z_s}\rho u_1\varphi\, dxdtdydzd\tau, 
\notag
\end{eqnarray}
which concludes the proof.
\end{proof}

We finally introduce the following notation-definition 
\begin{equation*}
\mathbb{F}_{0}^{1}=L^{2}(0,T;H_{0}^{1}(\Omega ))\times L^{2}(\Omega
_{T};L_{per}^{2}(\mathcal{T};H_{\rho }^{1}(Y_{m})))\times L^{2}(\Omega
_{T};L_{per}^{2}(Y\times \mathcal{T};H_{per}^{1}(Z_{s}))),
\end{equation*}%
where $H_{per}^{1}(Z_{s})$ stands for the space of functions $u\in
H_{per}^{1}(Z)$ with $\int_{Z_{s}}u(z)dz=0$. We similarly define $\mathcal{D}%
_{per}(Z_{s})$ and remark that the space $\mathbb{F}_{0}^{1}$ admits the
following dense subspace 
\begin{equation*}
\mathfrak{F}_{0}^{\infty }=\mathcal{C}_{0}^{\infty }(\Omega _{T})\times
\left( \mathcal{C}_{0}^{\infty }(\Omega _{T})\otimes \mathcal{D}_{per}(%
\mathcal{T})\otimes \mathcal{D}_{\rho }(Y_{m})\right) \times \left( \mathcal{%
C}_{0}^{\infty }(\Omega _{T})\otimes \mathcal{D}_{per}(Y\times \mathcal{T}%
)\otimes \mathcal{D}_{per}(Z_{s})\right) .
\end{equation*}%
Moreover, $\mathbb{F}_{0}^{1}$ is a Banach space under the norm

\begin{equation*}
\|(u_0,u_1,u_2)\|_{\mathbb{F}^1_0}=\|u_0\|_{L^2(0,T;H^1_0(\Omega))}+\|u_1%
\|_{L^2(\Omega_T;L^2(\mathcal{T};
H^1_\rho(Y_m)))}+\|u_2\|_{L^2(\Omega_T;L^2(Y\times\mathcal{T};
H^1_{per}(Z)))}.
\end{equation*}

\section{Homogenization process and main results\label{S4}}

Let $E$ be an ordinary sequence of real numbers $\varepsilon $ converging to
zero with $\varepsilon $.

\subsection{Derivation of the global limit problem\label{s4.1}}

Owing to Corollary~\ref{c1} and Theorem~\ref{t4}, there exist 
\begin{equation}  \label{eq28}
(u_0, u_1, u_2)\in \mathbb{F}^1_0
\end{equation}
and a subsequence $E^{\prime }$ of $E$ such that, as $E^{\prime
}\ni\varepsilon\to 0$, 
\begin{eqnarray}
P_\varepsilon u_\varepsilon & \to & u_0 \qquad \text{ in }\quad
L^2(0,T;H^1_0(\Omega))\text{-weak},  \notag \\
\notag \\
\frac{\partial (P_\varepsilon u_\varepsilon)}{\partial x_i} & %
\xrightarrow{w-ms} & \frac{\partial u_0}{ \partial x_i}+\frac{\partial u_1}{%
\partial y_i}+\frac{\partial u_2}{\partial z_i} \quad \text{ in }\quad
L^2(\Omega_T)\quad (1\leq j\leq N).  \notag
\end{eqnarray}
Moreover, as $E^{\prime }\ni\varepsilon\to 0$ we have 
\begin{equation*}
\rho^\varepsilon\chi_{\Omega^\varepsilon}\to \rho\chi_{Y_{m}}\otimes
\chi_{Z_{s}}\ \ \ \ \ \ \text{in }\ \ L^\infty(\Omega) \text{ weak-}*
\end{equation*}
so that Theorem \ref{t1} yields 
\begin{equation}  \label{eq281}
P_\varepsilon u_\varepsilon \to u_0 \qquad \text{ in }\quad
L^2(0,T;L^2(\Omega)).
\end{equation}
We are now in a position to formulate the first homogenization result.

\begin{theorem}
\label{t6} The triple $(u_{0},u_{1},u_{2})\in \mathbb{F}_{0}^{1}$ determined
above by \emph{(\ref{eq28})} is a solution to the following variational
problem:

\begin{equation}  \label{eq29}
\left\{ \begin{aligned} &(u_0, u_1, u_2)\in \mathbb{F}^1_0, \\&
\left(\int_{Y_m}\rho(y)dy\right) \int_{\Omega_T}\frac{\partial u_0}{\partial
t}\psi_0\,dxdt- \int_{\Omega_T}\left[\rho\chi_{Y_m}u_1(x,t),\frac{\partial
\psi_1}{\partial \tau}(x,t)\right]\,dxdt \\& \qquad\qquad = -
\iiint\limits_{\Omega_T\times Y_m\times \mathcal{T}}
G(y,\tau,u_0)\cdot\nabla \psi_0\,dxdtdyd\tau + \iiint\limits_{\Omega_T\times
Y_m\times \mathcal{T}}g(y,\tau,u_0)\psi_1\,dxdtdyd\tau\\& - \frac{1}{|Z_s|}\
\ \  \iiiint\limits_{\Omega_T\times Y_m\times
Z_s\times\mathcal{T}}A(y,\tau)(\nabla_x u_0 + \nabla_y u_1 + \nabla_z
u_2)\cdot(\nabla_x \psi_0 + \nabla_y \psi_1 + \nabla_z
\psi_2)\,dxdtdydzd\tau \\& - \frac{1}{|Z_s|}\ \ \
\iiiint\limits_{\Omega_T\times Y_m\times
Z_s\times\mathcal{T}}\left(\partial_r G(y,\tau,u_0)\cdot(\nabla_x u_0 +
\nabla_y u_1 + \nabla_z u_2)\right)\psi_0\,dxdtdydzd\tau \\& \text{ for all
} \ \ (\psi_0,\psi_1,\psi_2)\in \mathfrak{F}^\infty_0. \end{aligned} \right.
\end{equation}
\end{theorem}

\begin{proof}
Let $\varepsilon>0$ and let $(\psi_0,\psi_1,\psi_2)\in \mathfrak{F}^\infty_0$%
. The appropriate oscillating test function for our problem is defined as
follows: 
\begin{equation}  \label{eq30}
\psi_\varepsilon(x,t)=\psi_0(x,t)+\varepsilon\psi_1(x,t,\frac{x}{\varepsilon}%
,\frac{t}{\varepsilon^2})+\varepsilon^2\psi_2(x,t,\frac{x}{\varepsilon},%
\frac{x}{\varepsilon^2},\frac{t}{\varepsilon^2}),\qquad
(x,t)\in\Omega\times(0,T).
\end{equation}
Multiplying all terms in the main equation of (\ref{eq1}) by $%
\psi_\varepsilon(x,t)$ and integrating over $\Omega^\varepsilon\times (0,T)$
leads to: 
\begin{equation*}
\int_{\Omega^\varepsilon_T}\rho(\frac{x}{\varepsilon})\frac{\partial
u_\varepsilon}{\partial t}\psi_\varepsilon\, dxdt = -
\int_{\Omega^\varepsilon_T}A(\frac{x}{\varepsilon},\frac{t}{\varepsilon^2}%
)\nabla u_\varepsilon\cdot\nabla\psi_\varepsilon\, dxdt +\frac{1}{\varepsilon%
} \int_{\Omega^\varepsilon_T}g(\frac{x}{\varepsilon},\frac{t}{\varepsilon^2}%
,u_\varepsilon)\psi^\varepsilon\, dxdt,
\end{equation*}
or equivalently to 
\begin{eqnarray}  \label{eq31}
\int_{\Omega_T}\rho(\frac{x}{\varepsilon})\chi_{\Omega^\varepsilon}(\frac{x}{%
\varepsilon},\frac{x}{\varepsilon^2})\frac{\partial (P_\varepsilon
u_\varepsilon)}{\partial t}\psi_\varepsilon\, dxdt& = &-
\int_{\Omega_T}\chi_{\Omega^\varepsilon}(\frac{x}{\varepsilon},\frac{x}{%
\varepsilon^2})A(\frac{x}{\varepsilon},\frac{t}{\varepsilon^2})\nabla
(P_\varepsilon u_\varepsilon)\cdot\nabla\psi_\varepsilon\, dxdt \\
& &+\frac{1}{\varepsilon} \int_{\Omega_T}\chi_{\Omega^\varepsilon}(\frac{x}{%
\varepsilon},\frac{x}{\varepsilon^2})g(\frac{x}{\varepsilon},\frac{t}{%
\varepsilon^2},P_\varepsilon u_\varepsilon)\psi^\varepsilon\, dxdt.  \notag
\end{eqnarray}
We now pass to the limit in (\ref{eq31}) as $E^{\prime }\ni\varepsilon\to 0$
. We start with the term in the left hand side. We have 
\begin{equation}  \label{eq32}
\int_{\Omega_T}\rho(\frac{x}{\varepsilon})\chi_{\Omega^\varepsilon}(\frac{x}{%
\varepsilon},\frac{x}{\varepsilon^2})\frac{\partial (P_\varepsilon
u_\varepsilon)}{\partial t}\psi_\varepsilon\, dxdt =-\int_{\Omega_T}\rho(%
\frac{x}{\varepsilon})\chi_{\Omega^\varepsilon}(\frac{x}{\varepsilon},\frac{x%
}{\varepsilon^2})P_\varepsilon u_\varepsilon\frac{\partial \psi_\varepsilon}{%
\partial t}\, dxdt,
\end{equation}
and we recall that 
\begin{equation*}
\frac{\partial \psi_\varepsilon}{\partial t}(x,t)=\frac{\partial \psi_0}{%
\partial t}(x,t)+\varepsilon \frac{\partial \psi_1}{\partial t}(x,t,\frac{x}{%
\varepsilon},\frac{t}{\varepsilon^2})+\frac{1}{\varepsilon} \frac{\partial
\psi_1}{\partial \tau}(x,t,\frac{x}{\varepsilon},\frac{t}{\varepsilon^2})+
\varepsilon^2 \frac{\partial \psi_2}{\partial t}(x,t,\frac{x}{\varepsilon},%
\frac{x}{\varepsilon^2},\frac{t}{\varepsilon^2}) + \frac{\partial \psi_2}{%
\partial \tau}(x,t,\frac{x}{\varepsilon},\frac{x}{\varepsilon^2},\frac{t}{%
\varepsilon^2}).
\end{equation*}
It follows from (\ref{eq281}) that 
\begin{eqnarray}
\lim_{E^{\prime }\ni\varepsilon\to 0}\int_{\Omega_T}\rho(\frac{x}{\varepsilon%
})\chi_{\Omega^\varepsilon}(\frac{x}{\varepsilon},\frac{x}{\varepsilon^2}%
)P_\varepsilon u_\varepsilon\frac{\partial \psi_0}{\partial t}\, dxdt &=&
\iiint\limits_{\Omega_T\times Y\times
Z}\rho(y)\chi_{Y_m}(y)\chi_{Z_s}(z)u_0(x,t)\frac{\partial \psi_0}{\partial t}%
(x,t)\, dxdtdydz  \notag \\
&=&- |Z_s|\left(\int_{Y_m}\rho(y)dy\right)\int_{\Omega_T}\frac{\partial u_0}{%
\partial t}\psi_0\,dxdt .  \label{eq33}
\end{eqnarray}
By means of Proposition \ref{p4} we have

\begin{eqnarray}
&&\lim_{E^{\prime }\ni\varepsilon\to 0}\frac{1}{\varepsilon}%
\int_{\Omega_T}\rho(\frac{x}{\varepsilon})\chi_{\Omega^\varepsilon}(\frac{x}{%
\varepsilon},\frac{x}{\varepsilon^2})P_\varepsilon u_\varepsilon\frac{%
\partial \psi_1}{\partial \tau}(x,t,\frac{x}{\varepsilon},\frac{t}{%
\varepsilon^2})\,  \notag \\
& &\qquad\qquad\qquad\qquad = \iiint\limits_{\Omega_T\times Y\times Z\times 
\mathcal{T}}\rho(y)\chi_{Y_m}(y)\chi_{Z_s}(z)u_1\frac{\partial \psi_1}{%
\partial \tau}\, dxdtdydzd\tau  \label{eq34} \\
& & \qquad\qquad\qquad\qquad=|Z_s|\int_{\Omega_T}\left[\rho\chi_{Y_m}
u_1(x,t),\frac{\partial \psi_1}{\partial \tau}(x,t)\right]\,dxdt  \notag
\end{eqnarray}
The following limits hold: 
\begin{eqnarray}
& &\lim_{E^{\prime }\ni\varepsilon\to 0}\varepsilon\int_{\Omega_T}\rho(\frac{%
x}{\varepsilon})\chi_{\Omega^\varepsilon}(\frac{x}{\varepsilon},\frac{x}{%
\varepsilon^2})P_\varepsilon u_\varepsilon\frac{\partial \psi_1}{\partial t}%
(x,t,\frac{x}{\varepsilon},\frac{t}{\varepsilon^2})\, dxdt = 0,  \label{eq35}
\\
& & \lim_{E^{\prime }\ni\varepsilon\to 0}\varepsilon^2\int_{\Omega_T}\rho(%
\frac{x}{\varepsilon})\chi_{\Omega^\varepsilon}(\frac{x}{\varepsilon},\frac{x%
}{\varepsilon^2})P_\varepsilon u_\varepsilon\frac{\partial \psi_2}{\partial t%
}(x,t,\frac{x}{\varepsilon},\frac{x}{\varepsilon^2},\frac{t}{\varepsilon^2}%
)\, dxdt=0 ,  \label{eq36} \\
& &\lim_{E^{\prime }\ni\varepsilon\to 0}\int_{\Omega_T}\rho(\frac{x}{%
\varepsilon})\chi_{\Omega^\varepsilon}(\frac{x}{\varepsilon},\frac{x}{%
\varepsilon^2})P_\varepsilon u_\varepsilon\frac{\partial \psi_2}{\partial
\tau}(x,t,\frac{x}{\varepsilon},\frac{x}{\varepsilon^2},\frac{t}{%
\varepsilon^2})\, dxdt=0 .  \label{eq37}
\end{eqnarray}
After the passage to the limit in the left hand side of (\ref{eq37}), we
used the formula $\int_{\mathcal{T}}\frac{\partial \psi_2}{\partial\tau}%
d\tau=0$. Similar trivial arguments work for (\ref{eq35}) and (\ref{eq36}).
Thus as $E^{\prime }\ni\varepsilon\to 0$, we have 
\begin{eqnarray}
\int_{\Omega^\varepsilon_T}\rho(\frac{x}{\varepsilon})\frac{\partial
u_\varepsilon}{\partial t}\psi_\varepsilon\, dxdt &\to&
|Z_s|\left(\int_{Y_m}\rho(y)dy\right)\int_{\Omega_T}\frac{\partial u_0}{%
\partial t}\psi_0\,dxdt  \notag \\
&& \quad - |Z_s|\int_{\Omega_T}\left[\rho\chi_{Y_m}u_1(x,t),\frac{\partial
\psi_1}{\partial \tau}(x,t)\right]\,dxdt.  \label{eq38}
\end{eqnarray}

As regards the first term in the right hand side of (\ref{eq31}), it is
classical that, as $E^{\prime }\ni\varepsilon\to 0$, we have

\begin{eqnarray}  \label{eq39}
&& \int_{\Omega_T^\varepsilon}A(\frac{x}{\varepsilon},\frac{t}{\varepsilon^2}%
)\nabla (P_\varepsilon u_\varepsilon)\cdot\nabla\psi_\varepsilon\, dxdt
=\int_{\Omega_T}\chi_{\Omega^\varepsilon}(\frac{x}{\varepsilon},\frac{x}{%
\varepsilon^2})A(\frac{x}{\varepsilon},\frac{t}{\varepsilon^2})\nabla
(P_\varepsilon u_\varepsilon)\cdot\nabla\psi_\varepsilon\,
dxdt\qquad\qquad\qquad  \notag \\
&& \qquad \qquad\to \iiiint\limits_{\Omega_T\times Y_m\times Z_s\times 
\mathcal{T}}A(y,\tau)(\nabla_x u_0 + \nabla_y u_1 + \nabla_z u_2)(\nabla_x
\psi_0 + \nabla_y \psi_1 + \nabla_z \psi_2)\,dxdtdydzd\tau.
\end{eqnarray}

Concerning the second term in the right hand side of (\ref{eq31}), we first
rewrite it as follows: 
\begin{eqnarray}
\frac{1}{\varepsilon} \int_{\Omega_T^\varepsilon}g(\frac{x}{\varepsilon},%
\frac{t}{\varepsilon^2},P_\varepsilon u_\varepsilon)\psi_\varepsilon\, dxdt
&=& \frac{1}{\varepsilon} \int_{\Omega_T^\varepsilon}g(\frac{x}{\varepsilon},%
\frac{t}{\varepsilon^2},P_\varepsilon u_\varepsilon)\psi_0\, dxdt +
\int_{\Omega_T^\varepsilon}g(\frac{x}{\varepsilon},\frac{t}{\varepsilon^2}%
,P_\varepsilon u_\varepsilon)\psi_1(x,t,\frac{x}{\varepsilon},\frac{t}{%
\varepsilon^2})\, dxdt  \notag \\
&& + \varepsilon\int_{\Omega_T^\varepsilon}g(\frac{x}{\varepsilon},\frac{t}{%
\varepsilon^2},P_\varepsilon u_\varepsilon)\psi_2(x,t,\frac{x}{\varepsilon},%
\frac{x}{\varepsilon^2},\frac{t}{\varepsilon^2})\, dxdt.  \label{eq40}
\end{eqnarray}
It is straightforward from \cite[Lemma 5]{woukengaa} that 
\begin{eqnarray}  \label{eq41}
\int_{\Omega_T^\varepsilon}g(\frac{x}{\varepsilon},\frac{t}{\varepsilon^2}%
,P_\varepsilon u_\varepsilon)\psi_1(x,t,\frac{x}{\varepsilon},\frac{t}{%
\varepsilon^2})\, dxdt &=&\int_{\Omega_T}g(\frac{x}{\varepsilon},\frac{t}{%
\varepsilon^2},P_\varepsilon u_\varepsilon)\chi_{\Omega^\varepsilon}(\frac{x%
}{\varepsilon},\frac{x}{\varepsilon^2})\psi_1(x,t,\frac{x}{\varepsilon},%
\frac{t}{\varepsilon^2})\, dxdt  \notag \\
&\to &|Z_s|\iiint\limits_{\Omega_T\times Y_m\times\mathcal{T}%
}g(y,\tau,u_0)\psi_1(x,t,y,\tau)\, dxdtdyd\tau
\end{eqnarray}
as $E^{\prime }\ni\varepsilon\to 0$. Likewise, it holds that 
\begin{equation}  \label{eq42}
\lim_{E^{\prime }\ni\varepsilon\to 0}
\varepsilon\int_{\Omega_T^\varepsilon}g(\frac{x}{\varepsilon},\frac{t}{%
\varepsilon^2},P_\varepsilon u_\varepsilon)\psi_2(x,t,\frac{x}{\varepsilon},%
\frac{x}{\varepsilon^2},\frac{t}{\varepsilon^2})\, dxdt = 0.
\end{equation}
It then remains to deal with the first term in the right hand side of (\ref%
{eq40}). We have 
\begin{eqnarray}
\frac{1}{\varepsilon} \int_{\Omega_T^\varepsilon}g(\frac{x}{\varepsilon},%
\frac{t}{\varepsilon^2},P_\varepsilon u_\varepsilon)\psi_0\, dxdt &=&
-\int_{\Omega^\varepsilon_T}G(\frac{x}{\varepsilon},\frac{t}{\varepsilon^2}%
,P_\varepsilon u_\varepsilon)\cdot\nabla_x\psi_0\,dxdt  \notag \\
&&\qquad\qquad- \int_{\Omega^\varepsilon_T}\left(\partial_r G(\frac{x}{%
\varepsilon},\frac{t}{\varepsilon^2},P_\varepsilon
u_\varepsilon)\cdot\nabla_x u_\varepsilon\right)\psi_0\,dxdt.  \label{eq43}
\end{eqnarray}
It follows from \cite[Lemma 5 - Remark 2]{woukengaa} that as $E^{\prime
}\ni\varepsilon\to 0$, 
\begin{equation}  \label{eq44}
\int_{\Omega_T^\varepsilon}\left(\partial_r G(\frac{x}{\varepsilon},\frac{t}{%
\varepsilon^2},P_\varepsilon u_\varepsilon)\cdot\nabla_x
u_\varepsilon\right)\psi_0\,dxdt = \iiiint\limits_{\Omega_T\times Y_m\times
Z_s\times\mathcal{T}}\left(\partial_r G(y,\tau,u_0)\cdot (\nabla_x u_0 +
\nabla_y u_1 + \nabla_z u_2)\right)\psi_0\, dxdtdydzd\tau.
\end{equation}
Likewise, 
\begin{equation}  \label{eq45}
\lim_{E^{\prime }\ni\varepsilon\to 0}\int_{\Omega_T^\varepsilon}G(\frac{x}{%
\varepsilon},\frac{t}{\varepsilon^2},P_\varepsilon
u_\varepsilon)\cdot\nabla_x\psi_0\,dxdt = |Z_s|\iiint\limits_{\Omega_T\times
Y_m\times\mathcal{T}}G(y,\tau, u_0)\cdot\nabla_x \psi_0\, dxdtdyd\tau.
\end{equation}
Thus, as $E^{\prime }\ni\varepsilon\to 0$, we have 
\begin{eqnarray}
\frac{1}{\varepsilon} \int_{\Omega_T^\varepsilon}g(\frac{x}{\varepsilon},%
\frac{t}{\varepsilon^2},P_\varepsilon u_\varepsilon)\psi_\varepsilon\, dxdt
& \to &-\iiiint\limits_{\Omega_T\times Y_m\times Z_s\times\mathcal{T}%
}\left(\partial_r G(y,\tau,u_0)\cdot (\nabla_x u_0 + \nabla_y u_1 + \nabla_z
u_2)\right)\psi_0\, dxdtdydzd\tau  \notag \\
&& - \qquad |Z_s|\iiint\limits_{\Omega_T\times Y_m\times\mathcal{T}%
}G(y,\tau, u_0)\cdot\nabla_x \psi_0\, dxdtdyd\tau \\
&& + \qquad |Z_s|\iiint\limits_{\Omega_T\times Y_m\times\mathcal{T}%
}g(y,\tau,u_0)\psi_1(x,t,y,\tau)\, dxdtdyd\tau,  \notag
\end{eqnarray}
which combined with (\ref{eq38})-(\ref{eq39}) concludes the proof.
\end{proof}

The second term on the left hand side of (\ref{eq29}) needs further
investigations. In fact, for further needs, we would like to rewrite it
using formula (\ref{eq262}) of Proposition \ref{p3} as follows 
\begin{equation*}
- \int_{\Omega_T}\left
[\rho\chi_{Y_m}u_1(x,t),\frac{\partial \psi_1}{%
\partial \tau}(x,t)\right]\,dxdt = \int_{\Omega_T}\left[\rho\chi_{Y_m}\frac{%
\partial u_1}{\partial \tau}(x,t), \psi_1(x,t)\right]\,dxdt,
\end{equation*}
but this requires that $u_1\in \mathcal{W}$.

\begin{proposition}
\label{p6} The function $u_{1}\in L_{per}^{2}(\mathcal{T};H_{\rho
}^{1}(Y_{m}))$ defined by \emph{(\ref{eq28})} and Theorem \emph{\ref{t6}}
belongs to $\mathcal{W}$.
\end{proposition}

\begin{proof}
Let $\psi _{0}=\psi _{2}=0$ and $\psi _{1}=\varphi \otimes \psi $ in (\ref%
{eq29}), where $\varphi \in \mathcal{C}_{0}^{\infty }(\Omega _{T})$ and $%
\psi \in \mathcal{D}_{per}(\mathcal{T})\otimes \mathcal{D}_{\rho }(Y_{m})$.
Using the arbitrariness of $\varphi $, we are led to

\begin{eqnarray*}
&&-\left[ \rho \chi _{Y_{m}}u_{1}(x,t),\frac{\partial \psi }{\partial \tau }%
(x,t)\right] \,dxdt=\iint\limits_{Y\times \mathcal{T}}\chi _{Y_{m}}\rho 
\frac{\partial u_{1}}{\partial \tau }\psi \,dyd\tau \\
&&\qquad \qquad \qquad =\iint\limits_{Y\times \mathcal{T}}g(y,\tau
,u_{0})\psi \,dyd\tau -\frac{1}{|Z_{s}|}\iiint\limits_{Y_{m}\times
Z_{s}\times \mathcal{T}}A(y,\tau )(\nabla _{x}u_{0}+\nabla _{y}u_{1}+\nabla
_{z}u_{2})\nabla _{y}\psi \,dydzd\tau .
\end{eqnarray*}%
But $g=\text{div}_{y}G$, and the boundedness of the matrix $A$ implies that
the linear functional 
\begin{equation*}
\psi \mapsto \iint\limits_{Y\times \mathcal{T}}g(y,\tau ,u_{0})\psi
\,dyd\tau -\frac{1}{|Z_{s}|}\iiint\limits_{Y_{m}\times Z_{s}\times \mathcal{T%
}}A(y,\tau )(\nabla _{x}u_{0}+\nabla _{y}u_{1}+\nabla _{z}u_{2})\nabla
_{y}\psi \,dydzd\tau
\end{equation*}%
is continuous on $\mathcal{D}_{per}(\mathcal{T})\otimes \mathcal{D}_{\rho
}(Y_{m})$ with the $L_{per}^{2}(\mathcal{T};H_{\rho }^{1}(Y_{m}))$-norm.
Thus can be extended to and element of $L_{per}^{2}(\mathcal{T};(H_{\rho
}^{1}(Y_{m}))^{\prime })$. In other words, $\chi _{Y_{m}}\rho \frac{\partial
u_{1}}{\partial \tau }\in L_{per}^{2}(\mathcal{T};(H_{\rho
}^{1}(Y_{m}))^{\prime })$ and the proof is completed.
\end{proof}

Therefore the global homogenized problem of (\ref{eq1}) reads

\begin{equation}  \label{eq46}
\left\{ \begin{aligned} &(u_0, u_1, u_2)\in \mathbb{F}^1_0, \\&
\left(\int_{Y_m}\rho(y)dy\right) \int_{\Omega_T}\frac{\partial u_0}{\partial
t}\psi_0\,dxdt+ \int_{\Omega_T}\left[\rho\chi_{Y_m}\frac{\partial
u_1}{\partial \tau}(x,t),\psi_1(x,t)\right]\,dxdt \\& \qquad\qquad = -
\iiint\limits_{\Omega_T\times Y_m\times \mathcal{T}}
G(y,\tau,u_0)\cdot\nabla \psi_0\,dxdtdyd\tau + \iiint\limits_{\Omega_T\times
Y_m\times \mathcal{T}}g(y,\tau,u_0)\psi_1\,dxdtdyd\tau\\& - \frac{1}{|Z_s|}\
\ \  \iiiint\limits_{\Omega_T\times Y_m\times
Z_s\times\mathcal{T}}A(y,\tau)(\nabla_x u_0 + \nabla_y u_1 + \nabla_z
u_2)\cdot(\nabla_x \psi_0 + \nabla_y \psi_1 + \nabla_z
\psi_2)\,dxdtdydzd\tau \\& - \frac{1}{|Z_s|}\ \ \
\iiiint\limits_{\Omega_T\times Y_m\times
Z_s\times\mathcal{T}}\left(\partial_r G(y,\tau,u_0)\cdot(\nabla_x u_0 +
\nabla_y u_1 + \nabla_z u_2)\right)\psi_0\,dxdtdydzd\tau \\& \text{ for all
} \ \ (\psi_0,\psi_1,\psi_2)\in \mathfrak{F}^\infty_0. \end{aligned} \right.
\end{equation}
The variational problem (\ref{eq46}) is termed global since it contains the
macroscopic homogenized problem and the local problem.

\subsection{The macroscopic problem}

We are now in a position to derive the equation describing the macroscopic
behavior of the $\varepsilon $-problem (\ref{eq1}). The variational problem (%
\ref{eq46}) is equivalent to the following system:

\begin{equation}  \label{eq47}
\left\{ \begin{aligned} & \left(\int_{Y_m}\rho(y)dy\right)
\int_{\Omega_T}\frac{\partial u_0}{\partial t}\psi_0\,dxdt= -
\iiint\limits_{\Omega_T\times Y_m\times \mathcal{T}}
G(y,\tau,u_0)\cdot\nabla \psi_0\,dxdtdyd\tau \\& - \frac{1}{|Z_s|}\ \ \
\iiiint\limits_{\Omega_T\times Y_m\times
Z_s\times\mathcal{T}}A(y,\tau)(\nabla_x u_0 + \nabla_y u_1 + \nabla_z
u_2)\cdot(\nabla_x \psi_0 )\,dxdtdydzd\tau \\& - \frac{1}{|Z_s|}\ \ \
\iiiint\limits_{\Omega_T\times Y_m\times
Z_s\times\mathcal{T}}\left(\partial_r G(y,\tau,u_0)\cdot(\nabla_x u_0 +
\nabla_y u_1 + \nabla_z u_2)\right)\psi_0\,dxdtdydzd\tau \\& \text{ for all
} \ \ \psi_0\in \mathcal{C}^\infty_0(\Omega_T), \end{aligned} \right.
\end{equation}

\begin{equation}  \label{eq48}
\left\{ \begin{aligned} & \int_{\Omega_T}\left[\rho\chi_{Y_m}\frac{\partial
u_1}{\partial \tau}(x,t),\psi_1(x,t)\right]\,dxdt =
\iiint\limits_{\Omega_T\times Y_m\times
\mathcal{T}}g(y,\tau,u_0)\psi_1\,dxdtdyd\tau\\& - \frac{1}{|Z_s|}\ \ \
\iiiint\limits_{\Omega_T\times Y_m\times
Z_s\times\mathcal{T}}A(y,\tau)(\nabla_x u_0 + \nabla_y u_1 + \nabla_z
u_2)\cdot( \nabla_y \psi_1)\,dxdtdydzd\tau \\& \text{ for all }\psi_1\in
\mathcal{C}^\infty_0(\Omega_T)\bigotimes\mathcal{D}_{per}(\mathcal{T})%
\bigotimes\mathcal{D}_{\rho}(Y_m), \end{aligned} \right.
\end{equation}
and

\begin{equation}  \label{eq49}
\left\{ \begin{aligned} & \iiiint\limits_{\Omega_T\times Y_m\times
Z_s\times\mathcal{T}}A(y,\tau)(\nabla_x u_0 + \nabla_y u_1 + \nabla_z
u_2)\cdot( \nabla_z \psi_2)\,dxdtdydzd\tau = 0 \\& \text{ for all }\psi_2\in
\mathcal{C}^\infty_0(\Omega_T)\bigotimes\mathcal{D}_{per}(\mathcal{T})%
\bigotimes\mathcal{D}_{per}(Y_m)\bigotimes \mathcal{D}_{per}(Z).
\end{aligned} \right.
\end{equation}

We first deal with (\ref{eq49}). We start with a few preliminaries. We
define $H^1_{\#}(Z_s)$ to be the space of functions in $H^1(Z_s)$ assuming
same values on the opposites faces of $Z$, and satisfying $%
\int_{Z_s}u(z)dz=0 $. We remark that if $u\in H^1_{per}(Z)$ with $%
\int_{Z_s}u(z)dz=0$ then its restriction to $Z_s$ (which is still denoted by 
$u$ in the sequel) belongs to $H^1_{\#}(Z_s)$. Vice-versa, the extension of
a function $u\in H^1_{\#}(Z_s)$ belongs to $H^1_{per}(Z)$ with $%
\int_{Z_s}u(z)dz=0$. We have the following result whose proof is obvious and
therefore omitted.

\begin{proposition}
Let $1\leq j\leq N$ and let $(y,\tau)\in \mathbb{R}^N\times \mathbb{R}$ be
fixed. The following microscopic local problem admits a solution which is
uniquely defined almost everywhere in $Z_s$. 
\begin{equation}  \label{eq50}
\left\{ \begin{aligned} &\chi^j(y,\tau)\in H^1_{\#}(Z_s): \\&
\int_{Z_s}A(y,\tau) \nabla_z \chi^j\cdot \nabla_z \omega\,dz = -\sum_{k=1}^N
a_{kj}\int_{Z_s}\frac{\partial \omega}{\partial z_k}\, dz \\& \text{ for all
}\omega\in H^1_{\#}(Z_s). \end{aligned} \right.
\end{equation}
\end{proposition}

Back to (\ref{eq49}), let $\psi _{2}=\varphi \otimes \omega $ with $\varphi
\in \mathcal{D}(\Omega _{T})\otimes \mathcal{D}_{per}(\mathcal{T})\otimes 
\mathcal{D}_{\rho }(Y_{m})$ and $\omega \in \mathcal{D}_{per}(Z_{s})$. We get

\begin{equation*}
\iiint\limits_{\Omega_T\times Y_m\times \mathcal{T}}\varphi(x,t,y,\tau)\,
dxdt\left[\int_{Z_s}A(y,\tau)(\nabla_x u_0 + \nabla_y u_1 + \nabla_z
u_2)\cdot\nabla_z\omega\, dz\right]=0
\end{equation*}
which by the arbitrariness of $\varphi$ gives, for fixed $(x,t)\in\Omega_T$
and fixed $(y,\tau)\in \mathbb{R}^N\times\mathbb{R} $,

\begin{equation}  \label{eq51}
\int_{Z_s}A(y,\tau)\nabla_z u_2\cdot\nabla_z\omega\,
dz=-\int_{Z_s}A(y,\tau)(\nabla_x u_0 + \nabla_y u_1)\cdot\nabla_z\omega\, dz.
\end{equation}
By inspection of the microscopic problems (\ref{eq50}) and (\ref{eq51}) it
appears by the uniqueness of the solution to (\ref{eq50}) that for almost
all $(x,t,y,\tau)$ fixed in $\Omega_T\times\mathbb{R}^N\times \mathbb{R}$

\begin{equation}  \label{eq52}
u_2(x,t,y,\tau)=\sum_{j=1}^N\left(\frac{\partial u_0}{\partial x_j}(x,t)+%
\frac{\partial u_1}{\partial y_j}(x,t,y,\tau)\right)\chi^j(y,\tau)\quad 
\text{ a.e., \ \ in } Z_s.
\end{equation}
For further needs, we introduce a notation. We define the matrix $%
\nabla_z\chi$ by 
\begin{equation*}
\left(\nabla_z \chi\right)_{ij}=\frac{\partial \chi^i}{\partial z_j}\qquad (1\leq i,j \leq N).
\end{equation*}
Then we can write in short for almost all $(x,t,y,\tau,z)\in\Omega_T\times%
\mathbb{R}^N\times\mathbb{R}\times \mathbb{R}^N$, 
\begin{equation}  \label{eq53}
\nabla_z u_2 = \nabla_z\chi\cdot(\nabla_x u_0 + \nabla_y u_1).
\end{equation}

We can now proceed and look at the mesoscopic scale. Let $(x,t)\in\Omega_T$
and $(r,\xi)\in \mathbb{R}\times\mathbb{R}^N$ be freely fixed and let $%
\pi(x,t,r,\xi)$ be defined by the mesoscopic cell problem:

\begin{equation}  \label{eq54}
\left\{ \begin{aligned} &\pi_1(x,t,r,\xi)\in \mathcal{W}\\
&\left[\rho\chi_{Y_m}\frac{\partial \pi_1}{\partial
\tau},\psi_1\right]\,dxdt- \frac{1}{|Z_s|}\ \ \  \iiint\limits_{ Y_m\times
Z_s\times\mathcal{T}}\left[A(I+\nabla_z\chi)\right](\xi + \nabla_y \pi_1
)\cdot( \nabla_y \psi_1 )\, dydzd\tau \\& \qquad = \iint\limits_{ Y_m\times
\mathcal{T}}g(y,\tau,r)\psi_1\,dyd\tau \\& \text{ for all
}\psi_1\in\mathcal{W}, \end{aligned} \right.
\end{equation}
Where $I$ stands for the identity matrix. For the sake of simplicity, we put 
$\tilde{A}=\int_{Z_s}A(I+\nabla_z\chi)\, dz$. The following Proposition
addresses the question of existence and uniqueness of the solution to the
variational problem (\ref{eq54}).

\begin{proposition}
\label{p6} The following local variational problem admits a solution which
is uniquely defined on $Y_m\times\tau$: 
\begin{equation}  \label{eq55}
\left\{ \begin{aligned} &\pi_1(x,t,r,\xi)\in \mathcal{W}\\
&\left[\rho\chi_{Y_m}\frac{\partial \pi_1}{\partial \tau},\psi_1\right]-
\frac{1}{|Z_s|}\  \iint\limits_{
Y_m\times\mathcal{T}}\tilde{A}(y,\tau)\nabla_y \pi_1 \cdot \nabla_y \psi_1
\, dyd\tau \\& \qquad = \iint\limits_{ Y_m\times
\mathcal{T}}g(y,\tau,r)\psi_1\,dyd\tau + \frac{1}{|Z_s|}\  \iint\limits_{
Y_m\times\mathcal{T}}\tilde{A}(y,\tau)\xi\cdot \nabla_y \psi_1 \, dyd\tau\\&
\text{ for all }\psi_1\in\mathcal{W}. \end{aligned} \right.
\end{equation}
\end{proposition}

\begin{proof}
It is clear from the boundedness of the bilinear form on the left hand side,
and the boundedness of the linear form on the right hand side of (\ref{eq55}%
) that (\ref{eq55}) admits at least one solution in $\mathcal{W}$. As for
the question of uniqueness, let $\pi _{1},\ \theta _{1}$ be two solutions to
(\ref{eq55}). Then $\zeta _{1}=\pi _{1}-\theta _{1}$ solves (\ref{eq55})
with zero right hand side. This yields: 
\begin{equation*}
\left[ \rho \chi _{Y_{m}}\frac{\partial \zeta _{1}}{\partial \tau },\zeta
_{1}\right] -\frac{1}{|Z_{s}|}\ \iint\limits_{Y_{m}\times \mathcal{T}}\tilde{%
A}(y,\tau )\nabla _{y}\zeta _{1}\cdot \nabla _{y}\zeta _{1}\,dyd\tau =0.
\end{equation*}%
But formula (\ref{eq262}) of Proposition \ref{p3} implies 
\begin{equation*}
\left[ \rho \chi _{Y_{m}}\frac{\partial \zeta _{1}}{\partial \tau },\zeta
_{1}\right] =0.
\end{equation*}%
We are left with 
\begin{equation*}
\iint\limits_{Y_{m}\times \mathcal{T}}\tilde{A}(y,\tau )\nabla _{y}\zeta
_{1}\cdot \nabla _{y}\zeta _{1}\,dyd\tau =0,
\end{equation*}%
which by the uniform ellipticity of the homogenized matrix $\tilde{A}$
yields $\nabla _{y}\zeta _{1}=0\text{ \ a.e. in }Y_{m}\times \mathcal{T}$.
Therefore there exists a function $h$ depending only on $\tau $ such that $%
\xi _{1}(y,\tau )=h(\tau )\text{ for almost every }(y,\tau )\in Y_{m}\times 
\mathcal{T}$. But, since $\xi _{1}\in L_{per}^{2}(\mathcal{T};H_{\rho
}^{1}(Y_{m}))$, we have 
\begin{equation}
0=\int_{Y_{m}}\rho (y)\xi _{1}(y,\tau )\,dy=h(\tau )\int_{Y_{m}}\rho
(y)\,dy\ \ \text{ a.e. in }\mathcal{T}.
\end{equation}%
Thus $h=0$ since $\int_{Y_{m}}\rho (y)\,dy\neq 0$. Therefore $\xi _{1}=0$
almost every where in $Y_{m}\times \mathcal{T}$.
\end{proof}

Taking in particular $r=u_{0}(x,t)$ and $\xi =\nabla u_{0}(x,t)$ with $(x,t)$
arbitrarily chosen in $\Omega _{T}$ and then choosing in (\ref{eq54}) the
particular test functions $\psi _{1}=\varphi (x,t)v_{1}$, with $\varphi \in 
\mathcal{C}_{0}^{\infty }(\Omega _{T})$ and $v_{1}\in \left( \mathcal{D}%
_{per}(\mathcal{T})\otimes \mathcal{D}_{\rho }(Y_{m})\right) $, and finally
comparing the resulting equation with (\ref{eq48}), it follows by means of
Proposition~\ref{p6} (bear in mind that $\mathcal{D}_{per}(\mathcal{T}%
)\otimes \mathcal{D}_{\rho }(Y_{m})$ is dense in $\mathcal{W}$), that for
almost every $(x,t)\in \Omega _{T}$ we have

\begin{equation}  \label{eq56}
u_1(x,t) = \pi_1(x,t,u_0(x,t),\nabla_x u_0(x,t))
\end{equation}
almost everywhere on $Y_m \times\mathcal{T}$. The linearity of the problem (%
\ref{eq55}) suggests a more flexible expression of its solution $\pi_1$. We
formulate the following variational problems

\begin{equation}  \label{eq57}
\left\{ \begin{aligned} &\omega_1(x,t,r)\in \mathcal{W}\\
&\left[\rho\chi_{Y_m}\frac{\partial \omega_1}{\partial \tau},\psi \right]-
\frac{1}{|Z_s|}\  \iint\limits_{
Y_m\times\mathcal{T}}\tilde{A}(y,\tau)\nabla_y \omega_1 \cdot \nabla_y \psi
\, dyd\tau = \iint\limits_{ Y_m\times \mathcal{T}}g(y,\tau,r)\psi\,dyd\tau
\\& \text{ for all }\psi\in\mathcal{W}. \end{aligned} \right.
\end{equation}
and 
\begin{equation}  \label{eq58}
\left\{ \begin{aligned} &\theta(x,t)=(\theta_i(x,t))_{1\leq i\leq N}\in
(\mathcal{W})^N\\ &\left[\rho\chi_{Y_m}\frac{\partial \theta_i}{\partial
\tau},\psi\right]- \frac{1}{|Z_s|}\  \iint\limits_{
Y_m\times\mathcal{T}}\tilde{A}(y,\tau)\nabla_y \theta_i \cdot \nabla_y \psi
\, dyd\tau = \frac{1}{|Z_s|}\sum_{k=1}^N\  \iint\limits_{
Y_m\times\mathcal{T}}\tilde{A}_{ik}(y,\tau) \frac{\partial \psi}{\partial
y_k} \, dyd\tau\\& \text{ for all }\psi\in\mathcal{W}, \end{aligned} \right.
\end{equation}
and leave to the reader to check that they admits solutions that are
uniquely defined on $Y_m\times\tau$ and satisfy 
\begin{equation}  \label{eq59}
\pi_1(x,t,r,\xi)(y,\tau) = \theta(x,t,y,\tau)\cdot\xi +
\omega_1(x,t,y,\tau,r).
\end{equation}
Hence, the same lines of reasoning as above yields: 
\begin{equation}  \label{eq60}
u_1(x,t,y,\tau)=\theta(x,t,y,\tau)\cdot\nabla_x u_0(x,t) +
\omega_1(x,t,y,\tau,u_0(x,t)).
\end{equation}

We are now in a position to formulate the strong form of the macroscopic
variational problem (\ref{eq47}). Substituting in (\ref{eq47}) the
expression of $\nabla_z u_2$ obtained in (\ref{eq53}), we have:

\begin{equation}  \label{eq61}
\left\{ \begin{aligned} & \left(\int_{Y_m}\rho(y)dy\right)
\int_{\Omega_T}\frac{\partial u_0}{\partial t}\psi_0\,dxdt= -
\iiint\limits_{\Omega_T\times Y_m\times \mathcal{T}}
G(y,\tau,u_0)\cdot\nabla \psi_0\,dxdtdyd\tau \\& - \frac{1}{|Z_s|}\ \ \
\iiint\limits_{\Omega_T\times
Y_m\times\mathcal{T}}\tilde{A}(y,\tau)(\nabla_x u_0 + \nabla_y
u_1)\cdot(\nabla_x \psi_0 )\,dxdtdyd\tau \\& - \frac{1}{|Z_s|}\ \ \
\iiint\limits_{\Omega_T\times Y_m\times\mathcal{T}}\left[\partial_r
G(y,\tau,u_0)\cdot\left(\left(\int_{Z_s}(I+\nabla_z
\chi)dz\right)\cdot(\nabla_x u_0 + \nabla_y
u_1)\right)\right]\psi_0\,dxdtdyd\tau \\& \text{ for all } \ \ \psi_0\in
\mathcal{C}^\infty_0(\Omega_T). \end{aligned} \right.
\end{equation}
We put $\tilde{B}(y,\tau)=\int_{Z_s}(I+\nabla_z \chi(y,\tau,z))dz \ \
((y,\tau)\in \mathbb{R}^N\times \mathbb{R})$ and use (\ref{eq60}) to get:

\begin{equation}  \label{eq62}
\left\{ \begin{aligned} & \left(\int_{Y_m}\rho(y)dy\right)
\int_{\Omega_T}\frac{\partial u_0}{\partial t}\psi_0\,dxdt= -
\int\limits_{\Omega_T }\left(\iint_{Y_m\times \mathcal{T}}\partial_r
G(y,\tau,u_0)\cdot\nabla_x u_0 \right)\psi_0\,dxdtdyd\tau \\& -
\frac{1}{|Z_s|}\ \ \  \int_{\Omega_T}\hat{A}\nabla_x u_0 \cdot \nabla_x
\psi_0\,dxdt- \frac{1}{|Z_s|}\ \ \  \int\limits_{\Omega_T
}\left(\iint_{Y_m\times\mathcal{T}}\tilde{A}(y,\tau) \nabla_y
\omega_1\,dyd\tau\right)\cdot\nabla_x \psi_0 \,dxdt \\& - \frac{1}{|Z_s|}\ \
\ \int\limits_{\Omega_T}\left\{\left(\iint_{Y_m\times\mathcal{T}}\partial_r
G(y,\tau,u_0)\cdot\left(\tilde{B}(y,\tau)(I+\nabla_y
\theta)\right)dyd\tau\right)\cdot\nabla_x u_0\right\}\psi_0\,dxdt \\&-
\frac{1}{|Z_s|}\int\limits_{\Omega_T}\left(\iint_{Y_m\times\mathcal{T}}%
\partial_r G(y,\tau,u_0)\cdot\left(\tilde{B}(y,\tau)(I+\nabla_y
\theta)\cdot\nabla_y \omega_1\right)dyd\tau\right)\psi_0\,dxdt \\& \text{
for all } \ \ \psi_0\in \mathcal{C}^\infty_0(\Omega_T), \end{aligned} \right.
\end{equation}
where 
\begin{equation*}
\hat{A}(x,t)=\iint_{Y_m\times\mathcal{T}}\tilde{A}(y,\tau)(I+\nabla_y
\theta(x,t,y,\tau))dyd\tau\quad ((x,t)\in\Omega_T).
\end{equation*}
Setting 
\begin{eqnarray*}
L_1(x,t,r)&=&\iint_{Y_m\times\mathcal{T}}\tilde{A}(y,\tau) \nabla_y
\omega_1(x,t,y,\tau,r)\,dyd\tau,  \notag \\
L_2(x,t,r) & = & \iint_{Y_m\times\mathcal{T}}\left[|Z_s|\partial_r
G(y,\tau,r) + \partial_r G(y,\tau,r)\cdot\left(\tilde{B}(y,\tau)(I+\nabla_y
\theta(x,t,y,\tau))\right)\right] dyd\tau,  \notag \\
L_3(x,t,r)&=&\iint_{Y_m\times\mathcal{T}}\partial_r G(y,\tau,r)\cdot\left(%
\tilde{B}(y,\tau)(I+\nabla_y \theta(x,t,y,\tau))\cdot\nabla_y
\omega_1(x,t,y,\tau,r)\right)dyd\tau,  \notag
\end{eqnarray*}

we are led to the following result.

\begin{theorem}
The function $u_{0}$ determined by \emph{(\ref{eq28})} and solution to the
variational problem \emph{(\ref{eq47})}, is the unique solution to the
following boundary value problem: 
\begin{equation}
\left\{ \begin{aligned}
&|Z_s|\left(\int_{Y_m}\rho(y)\,dy\right)\frac{\partial u_0}{\partial t} =
\text{div}\, \left(\hat{A}(x,t)\nabla u_0\right) + \text{div}\, L_1(x,t,u_0)
- L_2(x,t,u_0)\cdot \nabla u_0 - L_3(x,t,u_0) \text{ in } \Omega_T\\ & u_0
=0 \qquad \text{on }\ \ \partial\Omega\times (0,T)\\ &u_0(x,0)= u^0(x)
\qquad \text{in }\ \ \Omega. \end{aligned}\right.  \label{eq63}
\end{equation}
\end{theorem}

\begin{proof}
The claim that $u_0$ solves the problem (\ref{eq63}) has been proved above
and the uniqueness of the solution to (\ref{eq63}) follows from the fact
that the functions $L_i(x,t,\cdot)$ $(1\leq i \leq N)$ are Lipschitz. This
can be proved by mimicking the reasoning in \cite{AP10}
\end{proof}

We can now formulate the homogenization result for problem (\ref{eq1}).

\begin{theorem}
Assuming that the hypotheses \textbf{A1-A4} are in place and letting $%
u_{\varepsilon }$ ($\varepsilon >0$) be the unique solution to \emph{(\ref%
{eq1})}, we have, as $\varepsilon \rightarrow 0$, 
\begin{equation*}
u_{\varepsilon }\rightarrow u_{0}\quad \text{in }\ \ L^{2}(\Omega _{T}),
\end{equation*}%
where $u_{0}\in L^{2}(0,T;H_{0}^{1}(\Omega ))$ is the unique solution to 
\emph{(\ref{eq63})}.
\end{theorem}


\begin{thebibliography}{99}
\bibitem{Acerbi1992} E. Acerbi, V. Chiad\`{o} Piat, G. Dal Maso, D.
Percivale, An extension theorem from connected sets, and homogenization in
general periodic domains, Nonlinear Anal. \textbf{18} (1992) 481-496.

\bibitem{AB96} G. Allaire, M. Briane, Multiscale convergence and reiterated
homogenization, Proc.R. Soc.Edingurgh Sect. A \textbf{126} (1996) 297--342.

\bibitem{AP10} G. Allaire, A. Piatnitski, Homogenization of nonlinear
reaction-diffusion equation with large reaction term. Ann. Univ. Ferrara 
\textbf{56} (2010) 141--161.

\bibitem{AL83} H.W. Alt, S. Luckhaus, Quasilinear elliptic-parabolic
differential equations. Math. Z. \textbf{183} (1983) 311--341.

\bibitem{ADP05} M. Amar, A DallAglio and F. Paronetto, Homogenization of
forward -backward parabolic equations, Asymptot. Anal., \textbf{42}(2005)
123--132.

\bibitem{AGP07} B. Amaziane, M. Goncharenko, L. Pankratov, Homogenization of
a convection--diffusion equation in perforated domains with a weak
adsorption, Z. angew. Math. Phys. \textbf{58} (2007) 592--611.

\bibitem{DN04} P. Donato, A. Nabil, Homogenization of semilinear parabolic
equations in perforated domains,\ Chinese Ann. Math. Ser. B \textbf{25}
(2004) 143--156.

\bibitem{douanlaaa} H. Douanla, Two-scale convergence of elliptic spectral
problems with indefinite density function in perforated domains, Asymptot.
Anal. \textbf{81} (2013) 251--272.

\bibitem{DNW13} H. Douanla, G. Nguetseng, J.L. Woukeng, Incompressible
viscous Newtonian flow in a fissured medium of general deterministic type,
J. Math. Sci. (N.Y.) \textbf{191} (2013) 214--242.

\bibitem{showalter}U. Hornung and R.E. Showalter, Diffusion models for fractured media, J. Math. Anal. Appl. \textbf{147} (1990) 69-80.


\bibitem{jaroudi}M.E. Jaroudi, Homogenization of an incompressible viscous flow in a porous medium with double porosity, Z. Angew. Math. Phys. \textbf{61} (2010) 1053-1083.

\bibitem{JNS10} W. J\"{a}ger, M. Neuss-Radu, T. A. Shaposhnikova,
Homogenization limit for the diffusion equation with nonlinear flux
condition on the boundary of very thin holes periodically distributed in a
domain, in case of a critical size, Doklady Mathematics \textbf{82} (2010)
736--740.

\bibitem{NR02} A. K. Nandakumaran, M. Rajesh, Homogenization of a parabolic
equation in perforated domain with Dirichlet boundary condition, Proc.
Indian Acad. Sci. (Math. Sci.) \textbf{112} (2002) 425--439.

\bibitem{GW2007} G. Nguetseng and J.L. Woukeng, $\Sigma $-convergence of
nonlinear parabolic operators, Nonlin. Anal. TMA \textbf{66} (2007)
968--1004.

\bibitem{pankovpankova} A Pankov and T.E. Pankova: Nonlinear Evolution
equations with non-invertible operator coefficient at the derivative, Dokl.
Akad. Nauk Ukrainy \textbf{9} (1993) 18--20 (in Russian).

\bibitem{paronetto04} F. Paronetto: Homogenization of degenerate
elliptic-parabolic equations, Asymptot. Anal. \textbf{37} (2004) 21--56.

\bibitem{PR11} A. Piatnitski, V. Rybalko, Homogenization of boundary value
problems for monotone operators in perforated domains with rapidly
oscillating boundary conditions of fourier type, J. Math. Sci. \textbf{177}
(2011) 109--140.

\bibitem{SW11} M. Sango and J.L. Woukeng, Stochastic Sigma-convergence and
applications. Dyn. PDE \textbf{8} (2011) 261--310.


\bibitem{woukengaa} N. Svanstedt, J.L. Woukeng, Periodic homogenization of
strongly nonlinear reaction-diffusion equations with large reaction terms,
Appl. Anal. \textbf{92} (2013) 1357-1378.

\bibitem{CMA} J.L. Woukeng, Reiterated homogenization of nonlinear pseudo
monotone degenerate parabolic operators, Commun. Math. Anal. \textbf{9}
(2010) 98--129.

\bibitem{CPAA} J.L. Woukeng, $\Sigma $-convergence and reiterated
homogenization of nonlinear parabolic operators, Commun. Pure Appl. Anal. 
\textbf{9} (2010) 1753--1789.

\bibitem{Wou15} J.L. Woukeng, Multiscale nonlocal flow in a fractured porous
medium, Ann Univ Ferrara \textbf{61} (2015) 173--200.
\end{thebibliography}
\end{document}